\documentclass[11pt]{amsart}

\usepackage{amsmath,amssymb,amscd,epsfig,graphicx}

\input{prepictex}
\input{pictex}
\input{postpictex}

\def\arr(#1,#2)(#3,#4){\arrow <0.15cm> [0.25, 0.75] from {#1} {#2} to {#3} {#4}}

\begin{document}
\theoremstyle{plain}
\newtheorem{thm}{Theorem}[section]
\newtheorem*{thm*}{Theorem}
\newtheorem{prop}[thm]{Proposition}
\newtheorem*{prop*}{Proposition}
\newtheorem{lemma}[thm]{Lemma}
\newtheorem{cor}[thm]{Corollary}
\newtheorem*{conj*}{Conjecture}
\newtheorem*{cor*}{Corollary}
\newtheorem{defn}[thm]{Definition}
\newtheorem{cond}{Condition}
\theoremstyle{definition}
\newtheorem*{defn*}{Definition}
\newtheorem{rems}[thm]{Remarks}
\newtheorem*{rems*}{Remarks}
\newtheorem*{proof*}{Proof}
\newtheorem*{not*}{Notation}
\numberwithin{equation}{section}

\newcommand{\npartial}{\slash\!\!\!\partial}
\newcommand{\Heis}{\operatorname{Heis}}
\newcommand{\Solv}{\operatorname{Solv}}
\newcommand{\Spin}{\operatorname{Spin}}
\newcommand{\SO}{\operatorname{SO}}
\newcommand{\ind}{\operatorname{ind}}
\newcommand{\Index}{\operatorname{index}}
\newcommand{\ch}{\operatorname{ch}}
\newcommand{\rank}{\operatorname{rank}}
\newcommand{\abs}[1]{\lvert#1\rvert}
 \newcommand{\A}{{\mathcal A}}
        \newcommand{\D}{{\mathcal D}}\newcommand{\HH}{{\mathcal H}}
        \newcommand{\LL}{{\mathcal L}}
        \newcommand{\B}{{\mathcal B}}
        \newcommand{\K}{{\mathcal K}}
\newcommand{\oo}{{\mathcal O}}
         \newcommand{\PP}{{\mathcal P}}
        \newcommand{\s}{\sigma}
\newcommand{\al}{\alpha}
        \newcommand{\coker}{{\mbox coker}}
        \newcommand{\p}{\partial}
        \newcommand{\dd}{|\D|}
        \newcommand{\n}{\parallel}
\newcommand{\bma}{\left(\begin{array}{cc}}
\newcommand{\ema}{\end{array}\right)}
\newcommand{\bca}{\left(\begin{array}{c}}
\newcommand{\eca}{\end{array}\right)}
\def\clsp{\overline{\operatorname{span}}}
\def\T{\mathbb T}
\def\Aut{\operatorname{Aut}}

\def\sign{\operatorname{sign}}

\newcommand\RR{{\mathcal R}}
\newcommand{\sr}{\stackrel}
\newcommand{\da}{\downarrow}
\newcommand{\tD}{\tilde{\D}}

        \newcommand{\R}{\mathbb R}
        \newcommand{\C}{\mathbb C}
        \newcommand{\h}{\mathbf H}
\newcommand{\Z}{\mathbf Z}
\newcommand{\N}{\mathbf N}
\newcommand{\tto}{\longrightarrow}
\newcommand{\ben}{\begin{displaymath}}
        \newcommand{\een}{\end{displaymath}}
\newcommand{\be}{\begin{equation}}
\newcommand{\ee}{\end{equation}}

        \newcommand{\bean}{\begin{eqnarray*}}
        \newcommand{\eean}{\end{eqnarray*}}
\newcommand{\nno}{\nonumber\\}
\newcommand{\bea}{\begin{eqnarray}}
        \newcommand{\eea}{\end{eqnarray}}

\def\cross#1{\rlap{\hskip#1pt\hbox{$-$}}}
        \def\intcross{\cross{0.3}\int}
        \def\bigintcross{\cross{2.3}\int}

\newcommand{\supp}[1]{\operatorname{#1}}
\newcommand{\norm}[1]{\parallel\, #1\, \parallel}
\newcommand{\ip}[2]{\langle #1,#2\rangle}
\setlength{\parskip}{.3cm}
\newcommand{\nc}{\newcommand}
\nc{\nt}{\newtheorem} \nc{\gf}[2]{\genfrac{}{}{0pt}{}{#1}{#2}}
\nc{\mb}[1]{{\mbox{$ #1 $}}} \nc{\real}{{\mathbb R}}
\nc{\comp}{{\mathbb C}} \nc{\ints}{{\mathbb Z}}
\nc{\Ltoo}{\mb{L^2({\mathbf H})}} \nc{\rtoo}{\mb{{\mathbf R}^2}}
\nc{\slr}{{\mathbf {SL}}(2,\real)} \nc{\slz}{{\mathbf
{SL}}(2,\ints)} \nc{\su}{{\mathbf {SU}}(1,1)} \nc{\so}{{\mathbf
{SO}}} \nc{\hyp}{{\mathbb H}} \nc{\disc}{{\mathbf D}}
\nc{\torus}{{\mathbb T}}
\newcommand{\tk}{\widetilde{K}}
\newcommand{\boe}{{\bf e}}\newcommand{\bt}{{\bf t}}
\newcommand{\vth}{\vartheta}
\newcommand{\CGh}{\widetilde{\CG}}
\newcommand{\db}{\overline{\partial}}
\newcommand{\tE}{\widetilde{E}}
\newcommand{\tr}{\mbox{tr}}
\newcommand{\ta}{\widetilde{\alpha}}
\newcommand{\tb}{\widetilde{\beta}}
\newcommand{\txi}{\widetilde{\xi}}
\newcommand{\hV}{\hat{V}}
\newcommand{\IC}{\mathbf{C}}
\newcommand{\IZ}{\mathbf{Z}}
\newcommand{\IP}{\mathbf{P}}
\newcommand{\IR}{\mathbf{R}}
\newcommand{\IH}{\mathbf{H}}
\newcommand{\IG}{\mathbf{G}}
\newcommand{\CC}{{\mathcal C}}
\newcommand{\CS}{{\mathcal S}}
\newcommand{\CG}{{\mathcal G}}
\newcommand{\CL}{{\mathcal L}}
\newcommand{\CO}{{\mathcal O}}
\nc{\ca}{{\mathcal A}} \nc{\cag}{{{\mathcal A}^\Gamma}}
\nc{\cg}{{\mathcal G}} \nc{\chh}{{\mathcal H}} \nc{\ck}{{\mathcal
B}} \nc{\cl}{{\mathcal L}} \nc{\cm}{{\mathcal M}}
\nc{\cn}{{\mathcal N}} \nc{\cs}{{\mathcal S}} \nc{\cz}{{\mathcal
Z}} \nc{\cM}{{\mathcal M}}
\nc{\sind}{\sigma{\rm -ind}}
\newcommand{\la}{\langle}
\newcommand{\ra}{\rangle}
\newcommand{\cT}{{\mathcal T}}
\newcommand{\cC}{{\mathcal C}}
\newcommand{\cA}{{\mathcal A}}
\newcommand{\cH}{{\mathcal H}}
\newcommand{\cO}{{\mathcal O}}
\newcommand{\cK}{{\mathcal K}}
\newcommand{\Hom}{{\rm Hom}}
\renewcommand{\H}{{\mathbb H}}

\renewcommand{\labelitemi}{{}} 

\def\O{{\mathcal O}}
\def\m{{\mathfrak m}}
\def\P{{\mathbb P}}
\def\Q{{\mathbb Q}}
\def\C{{\mathbb C}}
\def\Z{{\mathbb Z}}
\def\N{{\mathbb N}}
\def\PGL{{\rm PGL}}
\def\PSL{{\rm PSL}}
\def\GL{{\rm GL}}
\newcommand{\ie}{{\it i.e.\/}\ }
\newcommand{\eg}{{\it e.g.\/}\ }
\newcommand{\cf}{{\it cf.\/}\ }

 \title{Modular Index Invariants of Mumford Curves}
      \author{A. Carey, M. Marcolli, A. Rennie}
\email{}

\begin{abstract}
We continue an investigation initiated by Consani--Marcolli of the relation 
between the algebraic geometry of $p$-adic Mumford curves and the 
noncommutative geometry of graph $C^*$-algebras associated to the 
action of the uniformizing $p$-adic Schottky group on the Bruhat--Tits tree.
We reconstruct invariants of Mumford curves related to valuations of generators of
the associated Schottky group, by developing a graphical theory for KMS weights 
on the associated graph $C^*$-algebra, and using modular index theory for 
KMS weights. We give explicit examples of the construction of graph weights
for low genus Mumford curves. We then show that the theta functions of
Mumford curves, and the induced currents on the Bruhat--Tits tree, define
functions that generalize the graph weights. We show that such inhomogeneous
graph weights can be constructed from spectral flows, and that one can reconstruct
theta functions from such graphical data. 
\end{abstract}

\maketitle

\section{Introduction}

Mumford curves generalize the Tate uniformization of elliptic curves and provide $p$-adic 
analogues of the uniformization of Riemann surfaces, \cite{Mum}.
The type of $p$-adic uniformization considered for these curves is a close analogue 
of the Schottky uniformization of complex Riemann 
surfaces, where instead of a Schottky group $\Gamma \subset \PSL_2(\C)$ acting on the Riemann sphere $\P^1(\C)$, one has a $p$-adic Schottky group acting on the boundary of the Bruhat--Tits tree
and on the Drinfeld $p$-adic upper half plane. 

The analogy between Mumford curves and Schottky uniformization of Riemann surfaces was
a key ingredient in the results of Manin on Green functions of Arakelov geometry in terms of
hyperbolic geometry \cite{Man}, motivated by the analogy with earlier results of Drinfeld--Manin for the case of $p$-adic Schottky groups \cite{DriMan}. Manin's computation of the Green function for a 
Schottky-uniformized Riemann surface in terms of geodesics lengths in the hyperbolic handlebody
uniformized by the same Schottky group provides a geometric interpretation of the 
missing ``fibre
at infinity" in Arakelov geometry in terms the tangle of bounded geodesics 
inside the hyperbolic
3-manifold. In order to make this result compatible with Deninger's description 
of the Gamma factors of $L$-functions as regularized determinants and with 
Consani's archimedean cohomology \cite{Cons}, this 
formulation of Manin was reinterpreted in terms of noncommutative 
geometry by Consani--Marcolli \cite{CM}. In particular, the model proposed in
\cite{CM} for the ``fibre at infinity" uses a noncommutative space which describes 
the action of
the Schottky group $\Gamma$ on its limit set  $\Lambda_\Gamma \subset \P^1(\C)$ 
via the crossed product $C^*$-algebra $C(\Lambda_\Gamma)\rtimes \Gamma$. 
This is a particular case of a 
Cuntz--Krieger algebra given by the graph $C^*$-algebra of the finite graph 
$\Delta_\Gamma/\Gamma$,  with $\Delta_\Gamma$ the Cayley graph of $\Gamma\simeq \Z^{*g}$.

Following the same analogy between Schottky uniformization of Riemann surfaces and $p$-adic
uniformization of Mumford curves, Consani and Marcolli extended their construction \cite{CM}
to the case of Mumford curves, \cite{CM1,CM2}. 
More interesting graph $C^*$-algebras appear
in the $p$-adic case than in the archimedean setting, 
namely the ones associated to the graph $\Delta_\Gamma/\Gamma$, 
which is the dual graph of the specialization of the 
Mumford curve and to $\Delta'_\Gamma/\Gamma$, which 
is the dual graph of the closed fibre of the minimal smooth model
of the curve.  After these results of Consani--Marcolli, the construction 
was further refined in \cite{CMRV} and extended to some classes of higher rank 
buildings generalizing the rank-one case of Schottky groups acting on Bruhat--Tits trees. 
The relation between Schottky uniformizations,
noncommutative geometry, and graph $C^*$-algebras was further analyzed in 
\cite{CLM,CorMa}. 

The main question in this approach is how much of the 
algebraic geometry of Mumford curves can be 
recovered by means of the noncommutative geometry of certain $C^*$-algebras
associated to the action of the Schottky group on the Bruhat--Tits tree, on its
limit set, and on the Drinfeld upper half plane, and conversely how much the noncommutative
geometry is determined by algebro-geometric information coming from the Mumford
curve. 

In this paper we analyze another aspect of this question, based on recent results on circle
actions on graph $C^*$-algebras and associated  KMS states and modular index theory,
\cite{CNNR}. In particular, we first show how numerical information like the Schottky invariants
given by the translation lengths of a given set of generators of the Schottky group can be 
recovered from the modular index invariants of the graph $C^*$-algebra determined by the
action of the $p$-adic Schottky group on the Bruhat--Tits tree. We then analyze the relation
between graph weights for this same graph $C^*$-algebra and theta functions of the Mumford 
curve. Unlike the previous results on Mumford curves and noncommutative geometry, which
concentrated on the use of the graph $C^*$-algebra associated to the {\em finite} graph
$\Delta'_\Gamma/\Gamma$  or $\Delta_\Gamma/\Gamma$, here we use the full infinite graph 
$\Delta_{\mathbb K}/\Gamma$, where $\Delta_{\mathbb K}$ is the Bruhat--Tits tree, with boundary at infinity $\partial \Delta_{\mathbb K}/\Gamma= X_\Gamma({\mathbb K})$, the ${\mathbb K}$-points of the Mumford curve, with ${\mathbb K}$ a finite extension of $\Q_p$. The fact of working with an infinite graph requires a more subtle analysis of the modular index theory and a setting for the graph weights, where the main information is located inside the finite graph $\Delta'_\Gamma/\Gamma$ and is propagated along the infinite trees in $\Delta_{\mathbb K}/\Gamma$ attached to the vertices of $\Delta'_\Gamma/\Gamma$, towards the conformal boundary $X_\Gamma({\mathbb K})$. The graph weights are solutions of a combinatorial equation at the vertices of the graph, which can be thought of as governing a momentum flow through the graph. We prove that for graphs such as $\Delta_{\mathbb K}/\Gamma$ faithful graph weights are the same as gauge invariant, norm lower semicontinuous faithful semifinite functionals on the graph $C^*$-algebra. This is the basis for constructing KMS states associated to graph weights, which
are then used to compute modular index invariants for these type III geometries.

The theta functions of Mumford curves in turn can be described as in \cite{vdP} in terms of $\Gamma$-invariant currents on the Bruhat--Tits tree $\Delta_{\mathbb K}$ and corresponding signed measures of total mass zero on the boundary. We show that this description of theta functions leads to an inhomogeneous version of the equation defining graph weights, or equivalently to a homogeneous version, but where the weights are allowed to have a sign instead of being positive and are also required to be integer valued. We also show that there is an isomorphism between the abelian group of $\Gamma$-invariant
currents on the Bruhat-Tits tree and linear functionals on the $K_0$ of the graph $C^*$-algebra of the
quotient graph $C^*(\Delta_{\mathbb K}/\Gamma)$. This implies that theta functions of the Mumford curve define
functionals on the $K$-theory of the graph algebra, with the only ambiguity given by the action of ${\mathbb K}^*$.

Finally, we discuss how to use the spectral flow to construct solutions of the
inhomogeneous graph weights equations and how to use these to construct
theta functions in case where one has more than one (positive) graph weight
on $\Delta_{\mathbb K}/\Gamma$.

It would be interesting, in a similar manner, to explore how other invariants associated to
type III noncommutative geometries, such as the approach followed by Connes--Moscovici in
\cite{CoMo} and by Moscovici in \cite{Mosc}, may relate to the algebraic geometry of 
Mumford curves in the specific case of the algebras $C^*(\Delta_{\mathbb K}/\Gamma)$.

\section{Mumford Curves}

We recall here some well known facts from the theory of Mumford
curves, which we need in the rest of the paper. The results
mentioned in this brief introduction can be found for instance in
\cite{Mum}, \cite{Ma}, \cite{GvP} and were also reviewed in more detail in \cite{CM1}.

\subsection{The Bruhat--Tits tree}

Let ${\mathbb K}$ denote a finite extension of $\Q_p$ and let
$\O=\O_{\mathbb K}\subset {\mathbb K}$ be its ring of integers, with $\m\subset \O$ the
maximal ideal. The finite field $k = \O/\m$ of cardinality $q = \# \O/\m$ is
called the residue field. 

Let $\Delta^0_{\mathbb K}$ denote the set of equivalence classes of free rank 2
$\O$-modules, with the equivalence relation 
$$ M_1 \sim M_2  \Leftrightarrow \exists \lambda\in {\mathbb K}^*, \ \  M_1 =
\lambda M_2. $$ 
The group $\GL_2({\mathbb K})$ acts on $\Delta^0_{\mathbb K}$ by $gM =
\{gm~|~m\in M\}$. This descends to an action of $\PGL_2({\mathbb K})$, since
$M_1\sim M_2$ for $M_1$ and $M_2$ in the same ${\mathbb K}^*$-orbit.
Given $M_2\subset M_1$, one has
$M_1/M_2 \simeq \O/\m^l \oplus \O/\m^k$,
for some $l,k\in \N$. The action of
${\mathbb K}^*$ preserves the inclusion $M_2\subset M_1$, hence
one has a well defined metric 
\begin{equation}\label{dist}
d(M_1,M_2) = | l-k |.
\end{equation}
The Bruhat--Tits tree of $\PGL_2({\mathbb K})$ is the infinite graph with set
of vertices $\Delta^0_{\mathbb K}$, and an edge connecting two vertices $M_1$,
$M_2$ whenever $d(M_1,M_2) = 1$. It is an infinite tree with
vertices of valence $q+1$ where $q = \# \O/\m$. The group
$\PGL_2({\mathbb K})$ acts on $\Delta_{\mathbb K}$ by isometries. The boundary $\partial
\Delta_{\mathbb K}$ is naturally identified with $\P^1({\mathbb K})$. 

\subsection{$p$-adic Schottky groups and Mumford curves}\label{padicGamma}

A Schottky group $\Gamma$ is a finitely generated, discrete, torsion-free
subgroup of $\PGL_2({\mathbb K})$ 
whose nontrivial elements $\gamma\neq 1$ are all {\it
hyperbolic}, \ie such that the eigenvalues of $\gamma$ in ${\mathbb K}$ have
different valuation. The group $\Gamma$ acts freely on the tree $\Delta_{\mathbb K}$. 
Hyperbolic elements $\gamma$ have two fixed points $z^\pm(\gamma)$ 
on the boundary $\P^1({\mathbb K})$. 
For an element $\gamma\neq 1$ in $\Gamma$ the {\em axis} $L(\gamma)$ 
of $\gamma$ is the unique geodesic in the Bruhat--Tits tree $\Delta_{\mathbb K}$ 
with endpoints the two fixed points $z^\pm(\gamma) \in 
\P^1({\mathbb K})=\partial \Delta_{\mathbb K}$.

Let $\Lambda_\Gamma\subset \P^1({\mathbb K})$ be the closure  in $\P^1({\mathbb K})$
of the set of fixed points of the elements $\gamma\in\Gamma\setminus\{1\}$. 
This is called the {\em limit set} of $\Gamma$. 
Only in the case of $\Gamma = (\gamma)^\Z\simeq \Z$, with a single hyperbolic generator 
$\gamma$, one has $\# \Lambda_\Gamma < \infty$. This special case, as we
see below, correponds to Mumford curves of genus one. In general,
for higher genus (higher rank Schottky groups), the limit set is
uncountable (typically a fractal). The {\it domain of discontinuity}
for the Schottky group $\Gamma$ is the complement $\Omega_\Gamma({\mathbb K})
= \P^1({\mathbb K})\smallsetminus\Lambda_\Gamma$. The quotient $X_\Gamma :=
\Omega_\Gamma /\Gamma$ gives the analytic model (via uniformization) 
of an algebraic curve $X$ defined over ${\mathbb K}$ (\cf~\cite{Mum} p.~163).

For a $p$-adic Schottky group $\Gamma \subset \PGL(2,{\mathbb K})$ there is a smallest
subtree $\Delta'_\Gamma \subset \Delta_{\mathbb K}$ that contains the axes $L(\gamma)$ of all
the elements $\gamma\neq 1$ of $\Gamma$. The set of ends of $\Delta'_\Gamma$ in
$\P^1({\mathbb K})$ is the limit set $\Lambda_\Gamma$ of $\Gamma$. The group
$\Gamma$ acts on $\Delta'_\Gamma$ with quotient
$\Delta'_\Gamma /\Gamma$ a finite graph. There is also a smallest
tree $\Delta_\Gamma$ on which $\Gamma$ acts, with vertices a subset of vertices of the
Bruhat--Tits tree. The tree
$\Delta'_\Gamma$ contains extra vertices with respect to $\Delta_\Gamma$. 
These come from vertices of $\Delta_{\mathbb K}$ that are not vertices of $\Delta_\Gamma$, but
which lie on paths in $\Delta_\Gamma$. ($\Delta_\Gamma$ is not a subtree of $\Delta_{\mathbb K}$,
while $\Delta_\Gamma'$ is.) 
The quotient $\Delta_\Gamma/ \Gamma$ is also a finite graph. Both the graphs
$\Delta'_\Gamma /\Gamma$ and $\Delta_\Gamma/ \Gamma$ have
algebro-geometric significance: 
$\Delta'_\Gamma/\Gamma$ is the dual graph of the closed fibre of 
the minimal smooth model of the algebraic curve $X$ over ${\mathbb K}$;
$\Delta_\Gamma/\Gamma$ is the dual graph of the specialization of the curve $X$.
The latter is a $k$-split degenerate, stable curve, with $k$ the residue field of ${\mathbb K}$.

The set of ${\mathbb K}$-points $X_\Gamma({\mathbb K})$ of the Mumford curve is 
identified with the 
ends of the graph $\Delta_{\mathbb K}/\Gamma$. With a slight abuse of notation we sometimes 
say that $X_\Gamma$ is the boundary of $\Delta_{\mathbb K}/\Gamma$ and write
\begin{equation}\label{Mcurve}
 \partial \Delta_{\mathbb K}/\Gamma =X_\Gamma.
\end{equation}
The graph $\Delta_{\mathbb K}/\Gamma$ contains the finite subgraph
$\Delta'_\Gamma /\Gamma$. Infinite trees depart from the vertices of
the subgraph $\Delta'_\Gamma /\Gamma$ with ends on the boundary at infinity
$X_\Gamma$. We
assume that the base point $v$ belongs to $\Delta'_\Gamma$ so that
all these trees are oriented outward from the finite graph
$\Delta'_\Gamma /\Gamma$. All the nontrivial topology resides in the
graph $\Delta'_\Gamma /\Gamma$ from which one can read off
the genus of the curve.

So far, the use of methods of noncommutative geometry in the context
of Mumford curves and Schottky uniformization (\cf \cite{CM},
\cite{CM1}, \cite{CM2}, \cite{CLM}, \cite{CMRV}) concentrated on the
finite graphs $\Delta'_\Gamma /\Gamma$ and $\Delta_\Gamma/ \Gamma$
and noncommutative spaces associated to the dynamics of the action
of the Schottky group on its limit set. Here we consider the full
infinite graph $\Delta_{\mathbb K}/\Gamma$ of \eqref{Mcurve}. In fact, we
will show that it is precisely the presence in $\Delta_{\mathbb K}/\Gamma$ of
the infinite trees attached to the vertices of the finite subgraph
$\Delta'_\Gamma /\Gamma$ that makes it possible to construct
interesting KMS states on the associated graph $C^*$-algebra and
hence to apply the techniques of modular index theory to obtain new
invariants of a $K$-theoretic nature for Mumford curves.

\subsection{Directed graphs and their algebras}

There are different ways to introduce a structure of directed graph 
on the finite graphs $\Delta_\Gamma/\Gamma$, $\Delta_\Gamma'/\Gamma$
and on the infinite graph $\Delta_{\mathbb K}/\Gamma$. 

One possibility, considered for instance in \cite{CLM}, is not to
prescribe an orientation on the graphs. This means that one keeps for
each edge the choice of both possible orientations. The associated
directed graph has then double the number of edges to account for the two
possible orientations. This approach has the problem that it makes the
graph $C^*$-algebras more complicated and the combinatorics
correspondingly more involved than strictly necessary, so we will not
follow it here.

Another way to make the graphs of Mumford curves into directed graphs
is by the choice of a projective coordinate $z\in \P^1({\mathbb K})$ (\cf
\cite{CM1}, \cite{CM2}). The choice of the coordinate $z$ determines
uniquely a base point $v\in \Delta_{\mathbb K}^0$, given by the 
origin of three non-overlapping paths with ends the points $0$, $1$
and $\infty$ in $\P^1({\mathbb K})$. The choice of $v$ gives an orientation to
the tree $\Delta_{\mathbb K}$ given by the outward direction from $v$. This
gives an induced orientation to any fundamental domains of the
$\Gamma$-action in $\Delta_{\mathbb K}$, $\Delta_\Gamma'$ and $\Delta_\Gamma$
containing the base vertex $v$, which one can use to obtain all the
possible induced orientations on the quotient graphs.  

There is still another possibility of orienting the tree $\Delta_{\mathbb K}$ in
a way that is adapted to the action of $\Gamma$, and this is the one
we adopt here. It is described in Lemma \ref{defor} below.

Suppose we are given the choice of a
projective coordinate $z\in \P^1({\mathbb K})$ and assume that the corresponding
vertex $v$ is in fact a vertex of $\Delta_\Gamma'$. Let $\{ \gamma_1,
\ldots, \gamma_g \}$ be a set of generators for $\Gamma$. 
An orientation of $\Gamma\backslash \Delta_{\mathbb K}$ is then obtained from
a $\Gamma$-invariant orientation of $\Delta_{\mathbb K}$ as follows.

\begin{lemma}\label{defor}
Consider the chain of edges in $\Delta_{\mathbb K}$ connecting the base vertex
$v$ to $\gamma_i v$. Then there is a choice of orientation of these
edges that induces an orientation on the quotient graph and that
extends to a $\Gamma$-invariant orientation of $\Delta_{\mathbb K}$.  
\end{lemma}

\proof Consider the chain of edges between $v$ and $\gamma_i v$. 
If all of them have distinct images in the quotient graph orient 
them all in the 
direction away from $v$ and towards $\gamma_i v$. If there is more
than one edge in the path from $v$ to $\gamma_i v$ that maps to the
same edge in the quotient graph, orient the first one that occurs from
$v$ to $\gamma_i v$ and the others consistently with the induced
orientation of the corresponding edge in the quotient graph. 
Similarly, orient the edges between $v$ and $\gamma^{-1}_i v$ in the
direction pointing towards $v$, with the same caveat for edges with
the same image in the quotient. Propagate this orientation across the
tree $\Delta_\Gamma'$ by repeating the same procedure with the edges
between $\gamma_i^{\pm 1} v$ and $\gamma_j\gamma_i^{\pm 1} v$ and 
between $\gamma_i^{\pm 1} v$ and $\gamma_j^{-1}\gamma_i^{\pm 1} v$ 
and so on. Continuing in this way, one obtains an
orientation of the tree $\Delta_\Gamma'$ compatible with the induced
orientation on the quotient graph $\Delta_\Gamma'/\Gamma$. One then
orients the rest of the tree $\Delta_{\mathbb K}$ away from the subtree
$\Delta_\Gamma'$. 
\endproof

An example of the orientations obtained in this way on the tree
$\Delta_\Gamma$ and on the quotient graph for 
the genus two case is given in Figures \ref{trees1}, \ref{trees2},
\ref{trees3} below.  

\subsection{Graph algebras for Mumford curves}

For a more detailed introduction to graph
$C^*$-algebras we refer the reader to \cite{BPRS,kpr,R} and the
references therein. A directed graph $E=(E^0,E^1,r,s)$ consists of
countable sets $E^0$ of vertices and $E^1$ of edges, and maps
$r,s:E^1\to E^0$ identifying the range and source of each edge.
{\em We will always assume that the graph is row-finite},
which means that each vertex emits at most finitely many edges.
Later we will also assume that the graph is \emph{locally finite}
which means  it is row-finite and each vertex receives at most
finitely many edges. We write $E^n$ for the set of paths
$\mu=\mu_1\mu_2\cdots\mu_n$ of length $|\mu|:=n$; that is,
sequences of edges $\mu_i$ such that $r(\mu_i)=s(\mu_{i+1})$ for
$1\leq i<n$.  The maps $r,s$ extend to $E^*:=\bigcup_{n\ge 0} E^n$
in an obvious way. A \emph{loop} in $E$ is a path $L \in E^*$ with
$s ( L ) = r ( L )$, we say that a loop $L$ has an exit if there
is $v = s ( L_i )$ for some $i$ which emits more than one edge. If
$V \subseteq E^0$ then we write $V \ge w$ if there is a path $\mu
\in E^*$ with $s ( \mu ) \in V$ and $r ( \mu ) = w$
(we also sometimes say
 that $w$ is downstream from $V$). A \emph{sink}
is a vertex $v \in E^0$ with $s^{-1} (v) = \emptyset$, a
\emph{source} is a vertex $w \in E^0$ with $r^{-1} (w) =
\emptyset$.

A \emph{Cuntz-Krieger $E$-family} in a $C^*$-algebra $B$ consists
of mutually orthogonal projections $\{p_v:v\in E^0\}$ and partial
isometries $\{S_e:e\in E^1\}$ satisfying the \emph{Cuntz-Krieger
relations}
\begin{equation*}
S_e^* S_e=p_{r(e)} \mbox{ for $e\in E^1$} \ \mbox{ and }\
p_v=\sum_{\{ e : s(e)=v\}} S_e S_e^*   \mbox{ whenever $v$ is not
a sink.}
\end{equation*}
It is proved in \cite[Theorem 1.2]{kpr} that there is a universal
$C^*$-algebra $C^*(E)$ generated by a non-zero Cuntz-Krieger
$E$-family $\{S_e,p_v\}$.  A product
$S_\mu:=S_{\mu_1}S_{\mu_2}\dots S_{\mu_n}$ is non-zero precisely
when $\mu=\mu_1\mu_2\cdots\mu_n$ is a path in $E^n$. Since the
Cuntz-Krieger relations imply that the projections $S_eS_e^*$ are
also mutually orthogonal, we have $S_e^*S_f=0$ unless $e=f$, and
words in $\{S_e,S_f^*\}$ collapse to products of the form $S_\mu
S_\nu^*$ for $\mu,\nu\in E^*$ satisfying $r(\mu)=r(\nu)$ (cf.\
\cite[Lemma
  1.1]{kpr}).
Indeed, because the family $\{S_\mu S_\nu^*\}$ is closed under
multiplication and involution, we have
\begin{equation}
C^*(E)=\clsp\{S_\mu S_\nu^*:\mu,\nu\in E^*\mbox{ and
}r(\mu)=r(\nu)\}.\label{spanningset}
\end{equation}
The algebraic relations and the density of $\mbox{span}\{S_\mu
S_\nu^*\}$ in $C^*(E)$ play a critical role throughout the paper.
We adopt the conventions that vertices are paths of length 0, that
$S_v:=p_v$ for $v\in E^0$, and that all paths $\mu,\nu$ appearing
in (\ref{spanningset}) are non-empty; we recover $S_\mu$, for
example, by taking $\nu=r(\mu)$, so that $S_\mu S_\nu^*=S_\mu
p_{r(\mu)}=S_\mu$.

If $z\in S^1$, then the family $\{zS_e,p_v\}$ is another
Cuntz-Krieger $E$-family which generates $C^*(E)$, and the
universal property gives a homomorphism $\gamma_z:C^*(E)\to
C^*(E)$ such that $\gamma_z(S_e)=zS_e$ and $\gamma_z(p_v)=p_v$.
The homomorphism $\gamma_{\overline z}$ is an inverse for
$\gamma_z$, so $\gamma_z\in\Aut C^*(E)$, and a routine
$\epsilon/3$ argument using (\ref{spanningset}) shows that
$\gamma$ is a strongly continuous action of $S^1$ on $C^*(E)$. It
is called the \emph{gauge action}. Because $S^1$ is compact,
averaging over $\gamma$ with respect to normalised Haar measure
gives an expectation $\Phi$ of $C^*(E)$ onto the fixed-point
algebra $C^*(E)^\gamma$:
\[
\Phi(a):=\frac{1}{2\pi}\int_{S^1} \gamma_z(a)\,d\theta\ \mbox{ for
}\ a\in C^*(E),\ \ z=e^{i\theta}.
\]
The map $\Phi$ is positive, has norm $1$, and is faithful in the
sense that $\Phi(a^*a)=0$ implies $a=0$.

{}From Equation (\ref{spanningset}), it is easy to see that a graph
$C^*$-algebra is unital if and only if the underlying graph is
finite. When we consider infinite graphs, formulas which involve
sums of projections may contain infinite sums. To interpret these,
we use strict convergence in the multiplier algebra of $C^*(E)$:

\begin{lemma}[\cite{kpr}]\label{strict}
Let $E$ be a row-finite graph, let $A$ be a $C^*$-algebra
generated by a Cuntz-Krieger $E$-family $\{T_e,q_v\}$, and let
$\{p_n\}$ be a sequence of projections  in $A$. If $p_nT_\mu
T_\nu^*$ converges for every $\mu,\nu\in E^*$, then $\{p_n\}$
converges strictly to a projection $p\in M(A)$.
\end{lemma}

The directed graph $\Delta_{\mathbb K}/\Gamma$ 
we obtain from a Mumford curve, with the orientation of Lemma \ref{defor}, 
is locally finite, has no sources and contains a subgraph 
$\Delta_\Gamma'/\Gamma$ with no sources and with the following 
two properties. If $v$ is any vertex in $\Delta_{\mathbb K}/\Gamma$ there exists a 
path in $\Delta_{\mathbb K}/\Gamma$ with range $v$ and source contained in 
$\Delta_\Gamma'/\Gamma$, and for any path  with source outside
$M$, the range is outside $M$. 
For such a graph we can define a new circle action by restricting the 
gauge action to the subgraph. The properties of this action turn out to 
be crucial for us.

The reason may be found in \cite{CNNR}, where the existence of a Kasparov $A$-$A^\s$
module for a circle action $\s$ on $A$ was found to be equivalent to a condition on the
spectral subspaces $A_k=\{a\in A: \s_z(a)=z^ka\}$. The condition, called the {\em spectral 
subspace condition} in \cite{CNNR}, states that for all $k\in\Z$, $A_kA_k^*$, 
always an ideal in $A^\s$, is in fact a complemented ideal in $A^\s$. Thus we must have 
$A^\s=A_kA_k^*\oplus  G_k$ for some other ideal $G_k$.
It turns out that the graphs arising from Mumford curves allow us to define a circle action for 
which the spectral subspaces satisfy the spectral subspace condition.

\begin{defn}
Let $E$ be a locally finite directed graph with no sources, $M\subset E$ a 
subgraph with no sources and such that

\vspace{-6pt}

1) for any $v\in E^0$ there is a 
path $\mu$ with $s(\mu)\in M$ and $r(\mu)=v$
\vspace{-6pt}

2) for all paths 
$\rho$ with $s(\rho)\not\in M$ we have $r(\rho)\not\in M$. 
\vspace{-6pt}

Then we say that 
$E$ has zhyvot $M$, and that $M$ is a zhyvot of $E$.

The zhyvot action $\s:\T\to Aut(C^*(E))$ is defined by
$$ \s_z(S_e)=\left\{\begin{array}{ll} \gamma_z(S_e) & e\in M^1\\
S_e & e\not\in M^1\end{array}\right.\qquad \s_z(p_v)=p_v,\ \ v\in E^0,$$
where $\gamma$ is the usual gauge action.  If $\mu$ is a path in 
$E$, let $|\mu|_\s$ be the non-negative integer such that
$\s_z(S_\mu)=z^{|\mu|_\s}S_\mu$.
\end{defn}

{\bf Remark} The zhyvot of a graph need not be unique.

{\bf Example} In the case of Mumford curves, the finite graph $\Delta_\Gamma '/\Gamma$
gives a zhyvot for the infinite graph $\Delta_{\mathbb K}/\Gamma$. There are other possible 
choices of a zhyvot for the same graph $\Delta_{\mathbb K}/\Gamma$, which are
interesting from the point of view of the geometry of Mumford curves. 
In particular, in the theory of Mumford curves, one considers the
reduction modulo powers $\m^n$ of the maximal ideal $\m \subset
\O_{\mathbb K}$, which provides infinitesimal neighborhoods of order $n$ of
the closed fiber. For each $n\geq 0$, we consider a subgraph
$\Delta_{{\mathbb K},n}$ of the Bruhat-Tits tree $\Delta_{\mathbb K}$ defined by setting
$$ \Delta_{{\mathbb K},n}^0 := \{ v\in \Delta_{\mathbb K}^0 : \,  d(v,
\Delta_\Gamma ')\leq n \}, $$ with respect to the distance
\eqref{dist}, with $d(v,\Delta_\Gamma ') := \inf \{ d(v,\tilde v):
\, \tilde v \in (\Delta_\Gamma ')^0 \}$, and
$$ \Delta_{{\mathbb K},n}^1 :=\{ w\in \Delta_{\mathbb K}^1 : \, s(w), r(w)\in
\Delta_{{\mathbb K},n}^0 \}. $$ Thus, we have $\Delta_{{\mathbb K},0} = \Delta_\Gamma '$
and $\Delta_{\mathbb K}=\cup_n \Delta_{{\mathbb K},n}$. For all $n\in \N$, the graph
$\Delta_{{\mathbb K},n}$ is invariant under the action of the Schottky group
$\Gamma$ on $\Delta$, and the finite graph $\Delta_{{\mathbb K},n}/\Gamma$
gives the dual graph of the reduction $X_{\mathbb K} \otimes \O/ \m^{n+1}$.
Thus, we refer to the $\Delta_{{\mathbb K},n}$ as {\em reduction graphs}. They
form a directed family with inclusions $j_{n,m}: \Delta_{{\mathbb K},n}
\hookrightarrow \Delta_{{\mathbb K},m}$, for all $m\geq n$, with all the
inclusions compatible with the action of $\Gamma$. Each of the quotient graphs
$\Delta_{{\mathbb K},n}/\Gamma$ also gives a zhyvot for $\Delta_{\mathbb K}/\Gamma$.
In the following we will concentrate on the case where $M=\Delta_\Gamma '/\Gamma$
but one can equivalently work with the reduction graphs.

Given a graph $E$ with zhyvot $M$ and $k\geq 0$ define
$$ F_k:=\overline{\mbox{span}}\{S_\mu S_\nu^*:\,|\mu|_\s=
|\nu|_\s\geq k\},$$
$$ G_k:=
\overline{\mbox{span}}\{S_\mu S_\nu^*:\,0\leq |\mu|_\s=
|\nu|_\s< k,\ \mbox{and either} \ r(\mu)=r(\nu)\not\in M\ \mbox{or}\ r(\mu)=
r(\nu)\ \mbox{is a sink in}\ M\}.$$

\smallskip

Observe that in the definition of $G_k$, the sinks need not be sinks of the full graph $E$, 
just sinks of the subgraph $M$.

{\bf Notation} Given a path $\rho\in E^*$, we let $\underline{\rho}$ denote the initial segment
of $\rho$ and let $\overline{\rho}$ denote the final segment; in all cases the length of these 
segments will be clear from context. We always have $\rho=\underline{\rho}\overline{\rho}$.

\begin{lemma} Let $E$ be a locally finite directed graph with 
no sources and zhyvot $M\subset E$. Let  $F=C^*(E)^\s$ be the fixed
point algebra for the zhyvot action. Then
$$ F=F_k\oplus G_k,\qquad k=1,2,3,\dots.$$
\end{lemma}

\begin{proof}
We first check using generators that $F_kG_k=G_kF_k=\{0\}$; once we 
have shown that $F_k+G_k=F$ this will also show that
$F_k$ and $G_k$ are both ideals (that they are subalgebras follows 
from similar, but simpler, calculations to those below). 

Fix $k\geq 1$.
Let $S_\mu S_\nu^*\in G_k$ so that $0\leq|\mu|_\s=|\nu|_\s<k$ and either
$r(\mu)=r(\nu)\not\in M$ or is a sink of $M$. Let $S_\rho S_\tau^*\in F_k$ so that
$|\rho|_\s=|\tau|_\s\geq k$. Then
$$S_\mu S_\nu^* S_\rho S_\tau^*
=\left\{\begin{array}{ll} 
S_\mu S_{\overline\rho}S_\tau^*\delta_{\nu,\underline\rho} 
& |\nu|\leq |\rho|\\
S_\mu S_{\overline\nu}^*S_\tau^*\delta_{\underline\nu,\rho} 
& |\nu|\geq |\rho|
\end{array}\right.,
$$
where $|\cdot|$ denotes the usual length of paths. When $|\nu|\geq|\rho|$,
the product is nonzero if and only if $\underline{\nu}=\rho$ but
$$ |\underline{\nu}|_\s\leq|\nu|_\s<|\rho|_\s,$$
so this can not happen. When $|\nu|\leq|\rho|$,  the product is nonzero if and only if
$\nu=\underline\rho$, but the range of $\nu\not\in M$ or is a sink of $M$ while 
$|\nu|_\s=|\underline\rho|_\s$ implies that $|\underline\rho|_\s<|\rho|_\s$, 
and so $r(\underline\rho)\in M$ and is not a sink of $M$. Hence the product is zero, and 
$G_kF_k=\{0\}$. The computation $F_kG_k=\{0\}$ is entirely analogous, 
so we omit it.

To see that $F_k+G_k=F$, we need only show that the generators
$S_\mu S_\nu^*$ with $0\leq|\mu|_\s=|\nu|_\s<k$ and $r(\mu)=r(\nu)\in M$ is not a sink,
are sums of elements from $F_k$ and $G_k$, all other generators 
having been accounted for.

So let $0\leq n=|\mu|_\s=|\nu|_\s<k$ and recall that $|\rho|\preceq k$
if $|\rho|= k$ or $|\rho|<k$ and $r(\rho)$ is a sink. Then
$$ S_\mu S_\nu^*=\sum_{\rho\in E^*,\ s(\rho)=r(\mu),\ |\rho|\preceq k-n+1}
S_\mu S_\rho S_\rho^* S_\nu^*.$$
If $0\leq|\rho|_\s<k-n$ then we must have $r(\rho)\not\in M$ or $r(\rho)$ a sink of $M$.
This is because $|\rho|_\s\leq|\rho|$, and if $r(\rho)$ is not a sink, we have
strict inequality since $|\rho|=k-n+1$. Hence if $r(\rho)$ is not a sink, 
$r(\rho)\not\in M$. On the other hand if 
$r(\rho)$ is a sink of $E$, then either $r(\rho)\not\in M$ or $r(\rho)$ is a sink of $M$.

Thus for $0\leq|\rho|_\s<k-n$ we have
$S_\mu S_\rho S_\rho^*S_\nu^*\in G_k$, while if 
$k-n\leq|\rho|_\s\leq k-n+1$, we have 
$S_\mu S_\rho S_\rho^* S_\nu^*\in F_k$. 

Finally, to see that $F=F_k\oplus G_k$ for each 
$k\geq 0$, observe that we can split 
the sequence 
$$0\to F_k\stackrel{i}{\to}F\to G_k\to 0$$
using the homomorphism $\phi_k:F\to F_k$ defined by
$$\phi_k(f)=P_kfP_k,\qquad P_k=
\sum_{|\mu|_\s=k}S_\mu S_\mu^*.$$
Checking that $\phi_k$ is a homomorphism and has 
 range $F_k$ is an exercise with the generators.
\end{proof}

\begin{prop}\label{pr:global-assumption} 
Let $E$ be a locally finite directed graph without sources and 
with zhyvot $M$. For $k\in\Z$ 
let $A_k=\{a\in C^*(E):\s_z(a)=z^ka\}$ denote the spectral subspaces 
for the zhyvot action. Then 
$$ A_kA_k^*=\left\{\begin{array}{ll} F_k & k\geq 0\\ F & k\leq 0
\end{array}\right..$$
\end{prop}

{\bf Remark} In particular, the spectral subspace assumptions of \cite{CNNR} are 
satisfied for the zhyvot action on a graph with a zhyvot.

\begin{proof} With $|\mu|$ denoting the ordinary length of paths in $E$,
we have the product formula
\begin{equation}
(S_\mu S_\nu^*)(S_\s S_\rho^*)^*=S_\mu S_\nu^* S_\rho S_\s^*
=\left\{\begin{array}{ll} 
S_\mu S_{\overline\rho}S_\s^*\delta_{\nu,\underline\rho} & |\nu|\leq |\rho|\\
S_\mu S_{\overline\nu}^*S_\s^*\delta_{\underline\nu,\rho} & |\nu|\geq |\rho|
\end{array}\right.,
\label{eq:product}
\end{equation}
where $\underline{\rho}$ is the initial segment of $\rho$ of appropriate 
length, and $\overline{\rho}$ is the final segment. If $S_\mu S_\nu^*$, 
$S_\s S_\rho^*$ are in $A_k$, $k\geq 0$, then
$$ |\mu|_\s-|\nu|_\s=k=|\gamma|_\s-|\rho|_\s,$$
so that $|\gamma|_\s\geq k$ and $|\mu|_\s\geq k$. Together with Equation 
\eqref{eq:product}, this shows that for $k\geq 0$ we have $A_kA_k^*\in F_k$.
Conversely, if $S_\alpha S_\beta^*\in F_k$, so $|\alpha|_\s=|\beta|_\s\geq k$,
we can factor
$$ S_\alpha S_\beta^*=S_{\underline\alpha}S_{\overline\alpha}S_\beta^*
=S_{\underline\alpha}(S_\beta S_{\overline\alpha}^*)^*\in A_kA_k^*.$$

For $k\leq 0$ we of course have $A_kA_k^*\in F$, and so we need only 
show that for any $S_\alpha S_\beta^*\in F$, 
$S_\alpha S_\beta^*\in A_kA_k^*$.

Here we use the final property of zhyvot graphs, namely that we can 
find a path $\lambda\in E^*$ with $s(\lambda)\in M$ and $r(\lambda)
=r(\alpha)=r(\beta)$. Moreover, because $M$ has no sources, we can 
take $|\lambda|_\s$ as great as we like. Thus we can write
$$S_\alpha S_\beta^*=S_\alpha S_\lambda^*S_\lambda S_\beta^*
=(S_\alpha S_\lambda^*)(S_\beta S_\lambda^*)^*.$$
Choosing $|\lambda|_\s=|\alpha|_\s+|k|$ shows that $S_\alpha S_\beta^*
\in A_kA_k^*$.
\end{proof}

This allows us to recover some known structure of the fixed point algebra of a graph
algebra for the usual gauge action, and to understand via the spectral subspace 
condition of \cite{CNNR} exactly why the assumptions of \cite{PRen} were required 
to construct a Kasparov module.

\begin{cor} Let $E$ be a locally finite directed graph without sources. Then the 
fixed point algebra for the usual gauge action decomposes as
$$ F=F_k\oplus G_k,\ \ \ k=1,2,3,\dots$$
where
$$ F_k=\overline{{\rm span}}\{S_\mu S_\nu^*:\,|\mu|=
|\nu|\geq k\},\ \  G_k=
\overline{{\rm span}}\{S_\mu S_\nu^*:\,0\leq |\mu|=
|\nu|< k,\ r(\mu)=r(\nu)\ {\rm is\ a\ sink}\}.$$
\end{cor}

\begin{proof} This follows from Proposition \ref{pr:global-assumption} 
since $E$ is a graph with zhyvot $E$.
\end{proof}

\section{Schottky invariants of Mumford Curves and Field
Extensions}\label{fieldextsect}

\subsection{Schottky lengths and valuation}

Let $\Gamma\subset \PGL_2({\mathbb K})$ be a $p$-adic Schottky group acting by
isometries on the Bruhat--Tits tree $\Delta_{\mathbb K}$.  As recalled in \S \ref{padicGamma}
above, a hyperbolic element $\gamma\in\Gamma$ determines a unique axis
$L(\gamma)$ in $\Delta_{\mathbb K}$, which is the infinite path of edges
connecting the two fixed points $z^\pm(\gamma)\in
\Lambda_\Gamma\subset \P^1({\mathbb K})=\partial \Delta_{\mathbb K}$. The element $\gamma$
acts on $L(\gamma)$ by a translation of length $\ell(\gamma)$.

To a given set of generators $\{ \gamma_1, \cdots \gamma_g \}$ of
$\Gamma$ one can associate the translation lengths $\ell(\gamma_i)$.
We refer to the collection of values $\{ \ell(\gamma_i) \}$ as 
the {\em Schottky invariants} of $(\Gamma,\{\gamma_i\})$. 
For example, in the case of genus one, one can assume the generator of
$\Gamma$ is given by a matrix for the form
$$ \gamma= \left(\begin{array}{cc} q & 0 \\ 0 & 1 \end{array}\right) $$
with $|q|<1$ so that the fixed points are $z^+(\gamma)=0$ and
$z^-(\gamma)=\infty$. The element $\gamma$ acts on the axis
$L(\gamma)$ as a translation by a length $\ell(\gamma)=\log
|q|^{-1}=v_\m(q)$ equal to the number of vertices in the closed graph
(topologically a circle) $\Delta_\Gamma'/\Gamma$. We see clearly that,
even in the simple genus one case, knowledge of the Schottky invariant
$\ell(\gamma)$ does not suffice to recover the curve. This is clear
from the fact that the Schottky length only sees the valuation of
$q\in {\mathbb K}^*$. Nonetheless, the Schottky lengths give useful computable
invariants. 

\subsection{Field extensions}
In the following section, where we derive explicit KMS states
associated to the infinite graphs given by the quotients
$\Delta_{\mathbb K}/\Gamma$, we also discuss the issue of how the invariants
we construct in this way for Mumford curves behave under field
extensions of ${\mathbb K}$. To this purpose, we recall here briefly how the
graphs $\Delta_{\mathbb K}$ and $\Delta_\Gamma'$ are affected when passing to
a field extension (\cf \cite{Ma}). This was also recalled in more detail in \cite{CM1}.

Let ${\mathbb L}\supset {\mathbb K}$ be a field extension with finite degree,
$[{\mathbb L}:{\mathbb K}]<\infty$, and let $e_{{\mathbb L}/{\mathbb K}}$ be its
ramification index. Let $\O_{\mathbb L}$ and $\O_{\mathbb K}$ denote the respective rings of
integers. There is an embedding of the sets of vertices $\Delta_{\mathbb K}^0
\hookrightarrow \Delta_{\mathbb L}^0$ obtained by assigning to a free $\O_{\mathbb K}$-module 
$M$ of rank $2$ the free $\O_{\mathbb L}$-module of the same rank given by
$M\otimes_{\O_{\mathbb K}}\O_{\mathbb L}$. This operation 
preserves the equivalence relation.
However, the embedding $\Delta_{\mathbb K}^0
\hookrightarrow \Delta_{\mathbb L}^0$ obtained in this way is not
isometric, as one can see from the isomorphism
$(\O_{\mathbb K}/\m^r)\otimes \O_{\mathbb L} \simeq \O_{\mathbb L}/\m^{re_{{\mathbb L}/{\mathbb K}}}$.
This can be corrected by modifying the metric on the graphs $\Delta_{\mathbb L}$,
for all extensions ${\mathbb L}\supset {\mathbb K}$: if one uses 
the ${\mathbb K}$-normalized distance
\begin{equation}\label{Knormdist}
d_{\mathbb K}(M_1,M_2) := \frac{1}{e_{{\mathbb L}/{\mathbb K}}} d_{\mathbb L}(M_1,M_2),
\end{equation}
on $\Delta_{\mathbb L}^0$, one obtains an isometric
embedding $\Delta_{\mathbb K}^0 \hookrightarrow \Delta_{\mathbb L}^0$.

\medskip

Geometrically, the relation between the Bruhat--Tits trees
$\Delta_{\mathbb K}$ and $\Delta_{\mathbb L}$ is described by the following procedure
that constructs $\Delta_{\mathbb L}$ from $\Delta_{\mathbb K}$ given the values of
$e_{{\mathbb L}/{\mathbb K}}$ and $[{\mathbb L}:{\mathbb K}]$. The rule for inserting new vertices and edges when passing to a field extension ${\mathbb L} \supset {\mathbb K}$ is the following.
\begin{enumerate}
\item $e_{{\mathbb L}/{\mathbb K}}-1$ new vertices $\{ v_1, \ldots, v_{e_{{\mathbb L}/{\mathbb K}}-1} \}$
are inserted between each pair of adjacent vertices in $\Delta^0_{\mathbb K}$. Let $\Delta_{{\mathbb L}, {\mathbb K}}^0$ denote the set of all these additional vertices.
\item $q^f + 1$ edges depart from each vertex in
$\Delta_{\mathbb K}^0\cup \Delta_{{\mathbb L}, {\mathbb K}}^0$, with $f =
\frac{1}{e_{{\mathbb L}/{\mathbb K}}}[{\mathbb L}:{\mathbb K}]$. Each such edge has length
$\frac{1}{e_{{\mathbb L}/{\mathbb K}}}$.
\item Each new edge attached to a vertex in
$\Delta_{\mathbb K}^0\cup \Delta_{{\mathbb L}, {\mathbb K}}^0$ is the base of a number of
homogeneous tree of valence $q^f + 1$.  The number is determined by the property that 
in the resulting graph the vertex from which the trees stem also has to have valence $q^f+1$.
The Bruhat--Tits tree $\Delta_{\mathbb L}$ is the union of $\Delta_{\mathbb K}$ with the 
additional inserted vertices $\Delta_{{\mathbb L}, {\mathbb K}}^0$ and the added trees 
stemming from each vertex.
\end{enumerate}

This procedure is illustrated in Figure \ref{extTree}, which we report here from \cite{CM1}.

\begin{figure}
\begin{center}
\epsfig{file=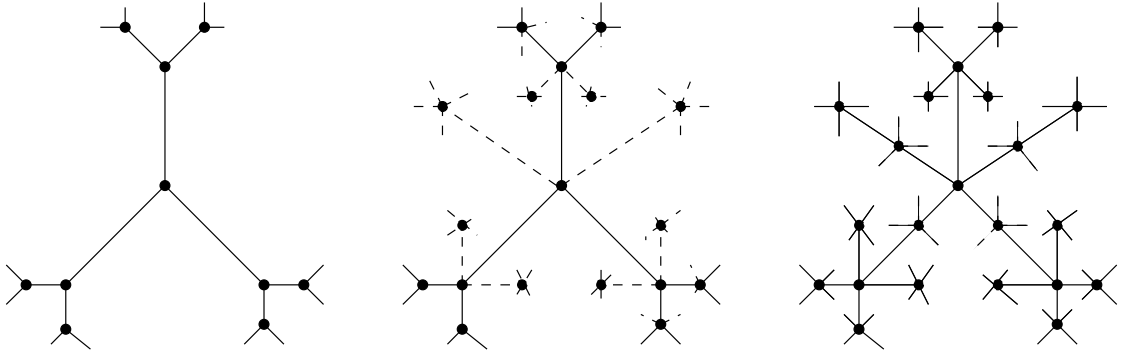} \caption{The tree $\Delta_{\mathbb K}$ for ${\mathbb K}=\Q_2$
and $\Delta_{\mathbb L}$ for a field extension with $f=2$ and $e_{{\mathbb L}/{\mathbb K}}=2$
\label{extTree}}
\end{center}
\end{figure}

Suppose we are given a $p$-adic Schottky group $\Gamma \subset \PGL_2({\mathbb K})$.
Since all nontrivial elements of $\Gamma$ are hyperbolic (the
eigenvalues have different valuation), one can see that the two
fixed points of any nontrivial element of $\Gamma$ are in
$\P^1({\mathbb K})=\partial \Delta_{\mathbb K}$. 
Thus, the limit set $\Lambda_\Gamma$ is contained in $\P^1({\mathbb K})$. 

When one considers a finite extension ${\mathbb L} \supset {\mathbb K}$ and the corresponding
Mumford curve $X_\Gamma({\mathbb L})=\Omega_\Gamma({\mathbb L})/\Gamma$ with
$\Omega_\Gamma({\mathbb L})=\P^1({\mathbb L})\smallsetminus \Lambda_\Gamma$, one can see
this as the boundary of the graph $\Delta_{\mathbb L}/\Gamma$. Notice that the
subtree $\Delta_{\Gamma,{\mathbb L}}'$ of $\Delta_{\mathbb L}$ and the subtree
$\Delta_{\Gamma,{\mathbb K}}'$ of $\Delta_{\mathbb K}$, both of which have boundary
$\Lambda_\Gamma$ only differ by the presence of the additional
$e_{{\mathbb L}/{\mathbb K}}-1$ new vertices in between any two adjacent vertices of
$\Delta_{\Gamma,{\mathbb K}}'$, while no new direction has been added (the
limit points are the same). In particular, this means that the
finite graph $\Delta_{\Gamma,{\mathbb L}}'/\Gamma$ is obtained from
$\Delta_{\Gamma,{\mathbb K}}'/\Gamma$ by adding $e_{{\mathbb L}/{\mathbb K}}-1$ vertices on each
edge. The infinite graph $\Delta_{\mathbb K}/\Gamma$ is obtained by adding to
each vertex of the finite graph $\Delta_{\Gamma,{\mathbb K}}'/\Gamma$ a finite
number (possibly zero) of infinite homogeneous trees of valence
$q+1$ with base at that vertex. Given the finite graph
$\Delta_{\Gamma,{\mathbb K}}'/\Gamma$, the number of such trees to be added at
each vertex is determined by the requirement that the valence of
each vertex of $\Delta_{\mathbb K}/\Gamma$ equals $q+1$. The infinite graph
$\Delta_{\mathbb L}/\Gamma$ is obtained from the graph $\Delta_{\mathbb K}/\Gamma$ by
replacing the homogeneous trees of valence $q+1$ starting from the
vertices of $\Delta_{\Gamma,{\mathbb K}}'/\Gamma$ with homogeneous trees of
valence $q^f+1$ stemming from the vertices of
$\Delta_{\Gamma,{\mathbb L}}'/\Gamma$, so that each resulting vertex of
$\Delta_{\mathbb L}/\Gamma$ has valence $q^f+1$.

We analyze the effect of field extensions from the point of view of KMS weights and
modular index theory in \S \ref{fieldext2sec} below.

\section{Graph KMS Weights on Directed Graphs}

Let $E$ be a row finite graph, and $C^*(E)$ the associated graph
$C^*$-algebra.

\begin{defn}\label{grapwdef} A  graph weight on $E$ is a pair of functions
  $g:E^0\to[0,\infty)$ and $\lambda:E^1\to[0,\infty)$
such that for all vertices $v$
$$ g(v)=\sum_{s(e)=v}\lambda(e)g(r(e)).$$
A graph weight is called faithful if $g(v)\neq 0$ for
all $v\in E^0$. If $\sum_{v\in E^0}g(v)=1$, we call $(g,\lambda)$ a
graph state.
\end{defn}

{\bf Remark} If $\lambda(e)=1$ for all $e\in E^1$, we obtain the
definition of a graph trace, \cite{T}.

{\bf Example} Suppose $e$ is a simple loop in a graph, with exit
$f$ at the vertex $v$, and that there are no other loops, and no other
exits from $v$, as in Figure \ref{loop-with-exit}. Set
$$ g(v)=\lambda(e) g(v)+\lambda(f) g(r(f)).$$
Then $g(v)=\frac{\lambda(f)}{1-\lambda(e)}g(r(f))$.

\begin{figure}
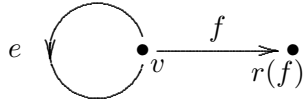

\[\qquad\qquad
\beginpicture

\setcoordinatesystem units <1cm,1cm>

\setplotarea x from 0 to 12, y from -0.9 to 0.5

\circulararc 325 degrees from 4.0 0.2 center at 3.4 0
\put{$\bullet$} at 4 0
\put{$\bullet$} at 6 0

\put{$e$} at 2.3 0
\put{$f$} at 5 0.3
\put{$v$} at 4.2 -0.2
\put{$r(f)$} at 5.8 -0.3

\arrow <0.25cm> [0.2,0.5] from 4.2 0 to 5.8 0

\arrow <0.25cm> [0.2,0.5] from 2.772 0.1 to 2.774 -0.1

\endpicture
\]
\caption{A loop with exit \label{loop-with-exit}}
\end{figure}

{\bf Remark} A graph weight is in fact specified by a single function
$h:E^*\to[0,\infty)$. For paths $v$ of length zero, i.e. vertices,
$h(v)=g(v)$ and for paths $\mu$ of length $k\geq 1$,
$h(\mu)=\lambda(\mu_1)\lambda(\mu_2)\cdots\lambda(\mu_k)$. We retain
the $(g,\lambda)$ notation but extend the definition of $\lambda$ by
$\lambda(\mu)=\prod_{i=1}^k\lambda(\mu_i)$.

Recall that a path $|\mu|$ has length $|\mu|\preceq k$ if $|\mu|=k$ or
$|\mu|<k$ and $r(\mu)$ is a sink.

We then have the following result, which can be proved by induction.

\begin{lemma} If $(g,\lambda)$ is a graph weight on $E$, then
$$ g(v)=\sum_{s(\mu)=v,\ |\mu|\preceq k}\lambda(\mu)g(r(\mu)),$$
where for a path $\mu=e_1\cdots e_j$, $j\leq k$,
$\lambda(\mu)=\prod\lambda(e_j)$.
\end{lemma}

We then define a functional $\phi_{g,\lambda}$ associated to a
graph weight $(g,\lambda)$ as follows.

\begin{defn} Given $(g,\lambda)$ on $E$ a graph weight, define
$\phi_{g,\lambda}:{\rm span}\{S_\mu S_\nu^*:\mu,\nu\in E^*\}\to\C$ by
$$ \phi_{g,\lambda}(S_\mu
S_\nu^*):=\delta_{\mu,\nu}\lambda(\nu)\phi_{g,\lambda}(p_{r(\nu)})
:=\lambda(\nu)\delta_{\mu,\nu}g(r(\nu)).$$
\end{defn}

This yields the following useful results.

\begin{prop}\label{pr:Hilbert-alg} 
Let $A_c={\rm span}\{S_\mu S_\nu^*:\mu,\nu\in
E^*\}$, and let $(g,\lambda)$ be a faithful graph weight on $E$. 
Then $A_c$ with the inner product
$$\la a,b\ra:=\phi_{g,\lambda}(a^*b)$$
is a modular Hilbert algebra (or Tomita algebra).
\end{prop}

\begin{proof} To complete the definition of modular Hilbert algebra,
we must supply a complex one parameter group of algebra automorphisms
$\s_z$ and verify a number of conditions set out in \cite{Ta}. So
for $z\in\C$ define
$$ \s_z(S_\mu
S_\nu^*)=\left(\frac{\lambda(\mu)}{\lambda(\nu)}\right)^zS_\mu
S_\nu^*.$$
Extending by linearity we can define $\s_z$ on all of $A_c$. To verify
the algebra automorphism property, it suffices to show that
$$\s_z(S_\mu S_\nu^* S_\rho S_\kappa^*)=\s_z(S_\mu S_\nu^*) \s_z(S_\rho
S_\kappa^*).$$
To do this we introduce some notation. If $\rho$ is a path we write
$\rho=\underline{\rho}\overline{\rho}$ where $\underline{\rho}$ is the
initial segment of $\rho$ (of length to be understood from context)
and $\overline{\rho}$ for the final segment.
First we compute the product on the left hand side.
$$ S_\mu S_\nu^* S_\rho S_\kappa^*=\left\{\begin{array}{ll}
\delta_{\nu,\underline{\rho}}S_\mu
S_{\overline{\rho}} S_\kappa^* & |\nu|\leq |\rho|\\
\delta_{\underline{\nu},\rho}S_\mu
S_{\overline{\nu}}^*S_\kappa^*&|\nu|\geq|\rho|\end{array}\right..$$
So
$$\s_z(S_\mu S_\nu^* S_\rho S_\kappa^*)= \left\{\begin{array}{ll}
\left(\frac{\lambda(\mu\overline{\rho})}{\lambda(\kappa)}\right)^z
\delta_{\nu,\underline{\rho}}S_\mu
S_{\overline{\rho}} S_\kappa^* & |\nu|\leq |\rho|\\
\left(\frac{\lambda(\mu)}{\lambda(\kappa\underline{\nu})}\right)^z
\delta_{\underline{\nu},\rho}S_\mu
S_{\overline{\nu}}^*S_\kappa^*&|\nu|\geq|\rho|\end{array}\right..$$
On the right hand side we have
\begin{align*}
\s_z(S_\mu S_\nu^*)\s_z( S_\rho S_\kappa^*) & =
\left(\frac{\lambda(\mu)\lambda(\rho)}{\lambda(\nu)\lambda(\kappa)}\right)^z
\left\{\begin{array}{ll} \delta_{\nu,\underline{\rho}}S_\mu
S_{\overline{\rho}} S_\kappa^*& |\nu|\leq|\rho|\\
\delta_{\underline{\nu},\rho} S_\mu S_{\overline{\nu}}^*S_\kappa^* &
|\nu|\geq |\rho|\end{array}\right.\nno
&=
\left\{\begin{array}{ll}
\left(\frac{\lambda(\mu)\lambda(\overline{\rho})}{\lambda(\kappa)}\right)^z
\delta_{\nu,\underline{\rho}}S_\mu
S_{\overline{\rho}} S_\kappa^*& |\nu|\leq|\rho|\\
\left(\frac{\lambda(\mu)}{\lambda(\overline{\nu})\lambda(\kappa)}\right)^z
\delta_{\underline{\nu},\rho} S_\mu S_{\overline{\nu}}^*S_\kappa^* &
|\nu|\geq |\rho|\end{array}\right.
\end{align*}
and this is easily seen to be the same as the left hand side whenever
the product is nonzero. Observe we have used the fact that
$$\lambda(\rho)=\lambda(\underline{\rho})\lambda(\overline{\rho}).$$

We need to show that $\la a,b\ra=\phi_{g,\lambda}(a^*b)$ does define an 
inner product. Let $a\in A_c$ and let $p\in A_c$ be a finite sum of vertex projections
such that $pa=ap=a$ ($p$ is a local unit for $a$). Then, since $g(v)>0$ for all $v\in E^0$, 
$$b\mapsto\frac{\phi_{g,\lambda}(pbp)}{\phi_{g,\lambda}(p)}$$
is a state on $pC^*(E)p$, and so positive. Hence
$$ \phi_{g,\lambda}(a^*a)\geq 0.$$

To show that the inner product is definite requires more care. First observe that 
if $\Psi:A_c\to \mbox{span}\{P_\mu=S_\mu S_\mu^*\}$ is the expectation on to 
the diagonal subalgebra, then 
$\phi_{g,\lambda}=\phi_{g,\lambda}\circ\Psi$. So we consider $a\in A_c$ and 
write $\Psi(a^*a)=\sum_\mu c_\mu P_\mu-\sum_\nu c_\nu P_\nu$. Here the 
$c_\mu,\,c_\nu>0$ and none of the paths $\mu$ is repeated in the sum. 
The average $\Psi(a^*a)$ is a positive
operator, so if $\Psi(a^*a)$ is
non-zero, all the $P_\nu$ in the negative part must be subprojections 
of $\sum_\mu P_\mu$ (otherwise $\Psi(a^*a)$ would have some negative spectrum). 
Since we are in a graph algebra, the Cuntz-Krieger relations tell us we can write
$$\sum_\mu P_\mu=\sum_\mu\sum_\rho P_{\mu\rho}$$
for some paths $\rho$ extending the various $\mu$, and that moreover all the $P_\nu$
appear as some $P_{\mu\rho}$. Thus
$$\Psi(a^*a)=\sum_\mu c_\mu P_\mu-\sum_\nu c_\nu P_\nu=
\sum_\mu c_\mu \sum_\rho P_{\mu\rho}-\sum_\nu c_\nu P_\nu=
\sum_\mu \sum_\rho d_{\mu\rho}P_{\mu\rho},$$
where the $d_{\mu\rho}$ are necessarily positive.
Now we can compute
$$\phi_{g,\lambda}(a^*a)=\phi_{g,\lambda}(\sum_\mu \sum_\rho d_{\mu\rho}P_{\mu\rho})
=\sum_\mu \sum_\rho d_{\mu\rho}\lambda(\mu\rho)g(r(\mu\rho))>0.$$

So now we come to verifying the various conditions defining a modular
Hilbert algebra. First, we need to consider the action of $A_c$ on
itself by left multiplication. This action is multiplicative,
$$\la ba,a\ra:=\phi_{g,\lambda}(a^*b^*a)=\la a,b^*a\ra,$$
and continuous
$$\la ba,ba\ra=\phi_{g,\lambda}(a^*b^*ba)\leq\Vert b^*b\Vert\la
a,a\ra,$$
where $\Vert\cdot\Vert$ denotes the $C^*$-norm coming from
$C^*(E)$. As $A_c^2=A_c$, the density of $A_c^2$ in $A_c$ is trivially
fulfilled. Also for all real $t$, $(1+\s_t)(S_\mu
S_\nu^*)=(1+(\lambda(\mu)/\lambda(\nu))^t)S_\mu S_\nu^*$, and so it is
an easy check to see that $(1+\s_t)(A_c)$ is dense in $A_c$ for all
real $t$. Also
$$ \la \s_{\bar{z}}(S_\mu S_\nu^*),S_\rho
S_\kappa^*\ra=\left(\frac{\lambda(\mu)}{\lambda(\nu)}\right)^{z}\la
S_\mu S_\nu^*,S_\rho S_\kappa^*\ra$$
is plainly analytic in $z$ (the reason for $\s_{\bar{z}}$ is that our
inner product is conjugate linear in the first variable). Since a
finite sum of analytic functions is analytic, $\la
\s_{\bar{z}}(a),b\ra$ is analytic for all $a,b\in A_c$.

The remaining items to check are the compatibility of $\s_z$ with the
inner product and involution, and all of these we can check for
monomials $S_\mu S_\nu^*$. The first item to check is
\begin{align*} \left(\s_z(S_\mu S_\nu^*)^*\right)&=
\left(\frac{\lambda(\mu)}{\lambda(\nu)}\right)^{\bar{z}}S_\nu S_\mu^*\nno
&=
\left(\frac{\lambda(\nu)}{\lambda(\mu)}\right)^{-\bar{z}}S_\nu S_\mu^*\nno
&= \s_{-\bar{z}}((S_\mu S_\nu^*)^*).
\end{align*}

Next we require $\la \s_z(a),b\ra=\la a,\s_{\bar{z}}(b)\ra$. So we
compute
\begin{align*} \la \s_z(S_\mu S_\nu^*),S_\rho S_\kappa^*\ra
&=\left(\frac{\lambda(\mu)}{\lambda(\nu)}\right)^{\bar{z}}g(r(\kappa))
\left\{\begin{array}{ll}
\delta_{\underline{\mu},\rho}
\delta_{\nu,\kappa\overline{\mu}}\lambda(\kappa\overline{\mu}) &
|\mu|\geq |\rho|\\
\delta_{\mu,\underline{\rho}}
\delta_{\nu\overline{\rho},\kappa}\lambda(\kappa) &
|\mu|\leq |\rho|\end{array}\right.\nno
&= g(r(\kappa))\left\{\begin{array}{ll}
\left(\frac{\lambda(\mu)}{\lambda(\kappa\overline{\mu})}\right)^{\bar{z}}
\delta_{\underline{\mu},\rho}
\delta_{\nu,\kappa\overline{\mu}}\lambda(\kappa\overline{\mu})
& |\mu|\geq|\rho|\\
\left(\frac{\lambda(\underline{\rho})}{\lambda(\nu)}\right)^{\bar{z}}
\delta_{\mu,\underline{\rho}}
\delta_{\nu\overline{\rho},\kappa}\lambda(\kappa) &
|\mu|\leq |\rho|\end{array}\right.\nno
&=g(r(\kappa))\left\{\begin{array}{ll}
\left(\frac{\lambda(\underline{\mu})}{\lambda(\kappa)}\right)^{\bar{z}}
\delta_{\underline{\mu},\rho}
\delta_{\nu,\kappa\overline{\mu}}\lambda(\kappa\overline{\mu})
& |\mu|\geq|\rho|\\
\left(\frac{\lambda(\underline{\rho})}{\lambda(\underline{\kappa})}\right)^{\bar{z}}
\delta_{\mu,\underline{\rho}}
\delta_{\nu\overline{\rho},\kappa}\lambda(\kappa) &
|\mu|\leq |\rho|\end{array}\right.\nno
&=g(r(\kappa))\left\{\begin{array}{ll}
\left(\frac{\lambda(\rho)}{\lambda(\kappa)}\right)^{\bar{z}}
\delta_{\underline{\mu},\rho}
\delta_{\nu,\kappa\overline{\mu}}\lambda(\kappa\overline{\mu})
& |\mu|\geq|\rho|\\
\left(\frac{\lambda(\underline{\rho})\lambda(\overline{\rho})}{\lambda(\underline{\kappa})\lambda(\overline{\rho})}\right)^{\bar{z}}
\delta_{\mu,\underline{\rho}}
\delta_{\nu\overline{\rho},\kappa}\lambda(\kappa) &
|\mu|\leq |\rho|\end{array}\right.\nno
&=\la S_\mu S_\nu^*,\s_{\bar{z}}(S_\rho S_\kappa^*)\ra,
\end{align*}
the last line following (when $|\mu|\leq|\rho|$) since the final
segments of $\rho$ and $\kappa$ must agree if the inner product is nonzero.
The final condition to check is that $\la \s_1(a^*),b^*\ra=\la
b,a\ra$.

First we compute
\begin{align*} \la \s_1(S_\mu S_\nu^*),S_\rho S_\kappa^*\ra &=
\frac{\lambda(\mu)}{\lambda(\nu)} \left\{\begin{array}{ll}
\delta_{\mu,\underline{\rho}} \delta_{\nu\overline{\rho},\kappa}
\lambda(\kappa) g(r(\kappa)) & |\mu|\leq|\rho|\\
\delta_{\underline{\mu},\rho} \delta_{\nu,\kappa\overline{\mu}}
\lambda(\kappa\overline{\mu}) g(r(\mu)) & |\mu|\geq
|\rho|\end{array}\right.\nno
&=\left\{\begin{array}{ll} \lambda(\rho)g(r(\kappa))
\delta_{\mu,\underline{\rho}} \delta_{\nu\overline{\rho},\kappa} &
|\mu|\leq|\rho| \\
\lambda(\mu)g(r(\mu))
\delta_{\underline{\mu},\rho} \delta_{\nu,\kappa\overline{\mu}} &
|\mu|\geq |\rho|\end{array}\right..
\end{align*}
Next we have
\begin{align*} \la S_\kappa S_\rho^*,S_\nu S_\mu^*\ra &=
\left\{\begin{array}{ll}
\delta_{\underline{\kappa},\nu}\delta_{\rho,\mu\overline{\kappa}}
\lambda(\mu\overline{\kappa}) g(r(\kappa)) & |\kappa|\geq|\nu|\\
\delta_{\kappa,\underline{\nu}} \delta_{\rho\overline{\nu},\mu}
\lambda(\mu) g(r(\mu)) & |\kappa|\leq|\nu|\end{array}\right.\nno
&=\left\{\begin{array}{ll} \lambda(\rho)g(r(\kappa))
\delta_{\underline{\kappa},\nu}\delta_{\rho,\mu\overline{\kappa}} &
|\kappa|\geq |\nu|\\
\lambda(\mu) g(r(\mu)) \delta_{\kappa,\underline{\nu}}
\delta_{\rho\overline{\nu},\mu} & |\kappa|\leq|\nu|\end{array}\right.
\end{align*}
Now for the inner product to be nonzero, we must have
$|\rho|+|\nu|=|\kappa|+|\mu|$, and so
$|\mu|\leq|\rho|\Leftrightarrow|\nu|\leq|\kappa|$. Comparing the
Kronecker deltas in the corresponding cases then yields the desired
equality for monomials, and the general case follows by linearity.
\end{proof}

\begin{thm}\label{thm:general-form}
Let $E$ be a locally finite directed graph. Then there is a one-to-one 
correspondence between gauge invariant norm lower semicontinuous faithful
semifinite functionals on $C^*(E)$ and faithful graph weights on $E$.
\end{thm}

\begin{proof} This is proved similarly to \cite[Proposition 3.9]{PRen} where the tracial case 
is considered. 

First  suppose that $(g,\lambda)$ is a faithful graph weight on $E$. Then 
$(A_c,\,\phi_{g,\lambda})$ is a modular Hilbert algebra.  
Since the left representation of $A_c$ on itself
is faithful, each $p_v$, $v\in E^0$, is represented by a non-zero projection. 
Let the representation be $\pi$.

The gauge invariance of $\phi_{g,\lambda}$ shows that for all $z\in\T$, the map 
$\gamma_z:A_c\to A_c$ extends to a unitary $U_z:\HH\to\HH$, where 
$\HH$ is the completion
of $A_c$ in the Hilbert space norm. It is easy to show that $U_z\pi(a)U_{\bar{z}}(b)=
\pi(\gamma_z(a))(b)$ for $a,\,b\in A_c$. Hence $U_z\pi(a)U_{\bar{z}}=
\pi(\gamma_z(a))$ and so $\alpha_z(\pi(a)):=U_z\pi(a)U_{\bar{z}}$ gives a 
point norm continuous action of $\T$ on $\pi(A_c)$ implementing the gauge action. 

We may thus invoke the gauge
invariant uniqueness theorem \cite{BPRS} 
to deduce that the representation extends to a faithful
representation of $C^*(E)$. 

Now $\pi(C^*(E))\subset \pi(A_c)''=\overline{\pi(A_c)}^{u.w.}$, the ultra-weak closure. 
Then
\cite[Theorem 2.5]{Ta} shows that the functional
$\phi_{g,\lambda}$ extends to a faithful, normal semifinite weight 
$\psi_{g,\lambda}$ on the 
left von Neumann algebra of $A_c$, $\pi(A_c)''$. 

Restricting the extension $\psi_{g,\lambda}$ to $C^*(E)$ gives a faithful weight. It is
norm semifinite since it is defined on $A_c$ which is dense in $C^*(E)$. Finally, if 
$a_j\to a$ in norm, then the $a_j$ converge ultra-weakly as well, so
$\lim\inf\psi_{g,\lambda}(a_j)\geq \psi_{g,\lambda}(a)$, which shows that the restriction
of $\psi_{g,\lambda}$ to $C^*(E)$ is norm lower semicontinuous. 

To get the gauge invariance of $\psi_{g,\lambda}$ we recall that $T\in\pi(A_c)''$ is 
in the domain of $\psi_{g,\lambda}$ if and only if $T=\pi(\xi)\pi(\eta)^*$ for 
left bounded elements $\xi,\,\eta\in\HH$. Then $\psi_{g,\lambda}(T)=
\psi_{g,\lambda}(\pi(\xi)\pi(\eta)^*):=\langle \xi,\eta\rangle$. As $U_z\xi$ and $U_z\eta$
are also left bounded we have
\begin{align*} \psi_{g,\lambda}(U_zTU_{\bar{z}})&=
\psi_{g,\lambda}(U_z\pi(\xi)\pi(\eta)^*U_{\bar{z}})=
\psi_{g,\lambda}(U_z\pi(\xi)(U_z\pi(\eta))^*)\\
&=\psi_{g,\lambda}(\pi(\gamma_z(\xi))\pi(\gamma_z(\eta))^*)=
\langle U_z\xi,U_z\eta\rangle\\
&=\langle \xi,\eta\rangle=\psi_{g,\lambda}(T).
\end{align*}
So $\psi_{g,\lambda}$ is $\alpha_z$ invariant, and $a\mapsto\psi_{g,\lambda}(\pi(a))$
defines a faithful semifinite norm lower semicontinuous gauge invariant weight on
$C^*(E)$.

Conversely, suppose that $\phi$ is a faithful semifinite norm lower 
semicontinuous weight on $C^*(E)$ which is gauge invariant. Define
$$g(v)=\phi(p_v),\qquad \lambda(e)=\frac{\phi(S_e S_e^*)}{\phi(S_e^*S_e)}.$$
It is readily checked that $(g,\lambda)$ is a faithful graph weight.
\end{proof}

In order to make contact with the index theory for KMS weights set out in
\cite{CNNR}, we require the action associated to our graph weight to be 
a circle action satisfying the spectral subspace condition, namely that $A_kA_k^*$ 
should be complemented 
in the fixed point algebra $F$.

A sufficient condition to obtain a circle action is that
$\lambda(e)=\lambda^{n_e}$ for every edge $e\in E^1$, where now
$n:E^1\to\Z$, and $\lambda\in (0,1)$.
In fact we will simplify matters further and deal here just with a
function of the form $n_e\in\{0,1\}$ for all $e\in E^1$. While this is rather restrictive, it
suffices for the examples we consider here. We call such functions {\em special 
graph weights}. In order for our special graph weight to accurately reflect the properties
of the zhyvot action on our graph, we will also require that $n_e=1$ if and only if $e\in M^1$. 

Also all the graphs we wish to consider are graphs with {\em finite} zhyvots, 
with the rest
of the graph being composed of 
trees. Since it is easy to construct faithful graph traces (i.e. special graph weights
with $n_e\equiv 0$) on (unions of)
trees given just the values of the trace on the root(s),
\cite{PRen}, it seems we need only worry about constructing a  graph state on the zhyvot. 

However, there is a subtlety: 
{\em neglecting the trees can affect the existence of special graph weights.}

{\bf Example} Graph  states on $SU_q(2)$. Recall that for $0\leq q<1$
the $C^*$-algebra $SU_q(2)$ is (isomorphic to) the graph $C^*$-algebra
of the  graph in Figure \ref{SUq2Fig}, \cite{HS}.

\begin{figure}
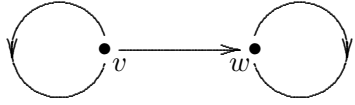

\[\qquad\qquad
\beginpicture

\setcoordinatesystem units <1cm,1cm>

\setplotarea x from 0 to 12, y from -0.9 to 0.5

\circulararc 325 degrees from 4.0 0.2 center at 3.4 0
\put{$\bullet$} at 4 0
\put{$\bullet$} at 6 0

\put{$v$} at 4.2 -0.2
\put{$w$} at 5.8 -0.2

\circulararc -325 degrees from 6 0.2 center at 6.6 0

\arrow <0.25cm> [0.2,0.5] from 4.2 0 to 5.8 0

\arrow <0.25cm> [0.2,0.5] from 7.228 0.1 to 7.226 -0.1

\arrow <0.25cm> [0.2,0.5] from 2.772 0.1 to 2.774 -0.1

\endpicture
\]
\caption{Graph of $SU_q(2)$ \label{SUq2Fig}}
\end{figure}

We want to solve
$$ g(v)=\lambda^{n_1}g(v)+\lambda^{n_2}g(w),\ \
g(w)=\lambda^{n_3}g(w).$$
First $\lambda^{n_3}=1$, so for $\lambda\neq 1$ (which we aren't
interested in), $n_3=0$. Then
$$ g(v)=\frac{\lambda^{n_2}}{1-\lambda^{n_1}}g(w).$$
Imposing the requirement that we have a graph state, $g(v)+g(w)=1$, we
get
$$ g(v)=\frac{\lambda^{n_2}}{1-\lambda^{n_1}+\lambda^{n_2}},\ \ \
g(w)=\frac{1-\lambda^{n_1}}{1-\lambda^{n_1}+\lambda^{n_2}} .$$
Observe that if $n_1=n_2$ we have
$$g(v)=1-\lambda^{n_1},\ \ g(w)=\lambda^{n_1}.$$
In this case we get the Haar state by setting $\lambda=q^{2/n_1}$, \cite{CRT}.

Observe that for $\lambda=1$ the only nonzero graph trace vanishes on
$w$, and we get the usual trace on the top circle with the kernel of 
$\phi_g=C(S^1)\otimes\K$.  For $\lambda>1$,
we get the same family as before by replacing $(n_1,n_2)$ by
$(-n_1,-n_2)$.
For a special graph state we must have $n_1=1$ and $n_3=0$. For $n_2$ we may choose 
either value.

So it seems we can not obtain a special graph weight with $n_e=1$ for all edges in the 
zhyvot. However, if we add trees to the graph, the loop on the vertex $w$ will acquire exits,
and then it is easy to construct special graph weights with $n_e=1$ precisely when 
$e$ is an edge in the zhyvot. 

\begin{lemma} Let $M$ be a finite graph and label the vertices $v_1,\dots,v_n$ 
so that the sinks, if any, are $v_{r+1},\dots,v_n$. Let $p_{jk}\in\N\cup\{0\}$ be the 
number of edges from $v_j$ to $v_k$.
Then $M$ has a faithful special 
graph state $(g,\lambda,n)$ 
for $\lambda\in(0,1)$ and $n:E^1\to \{1\}\subset \N$ if and only if the matrix
$$\begin{pmatrix} (\lambda p_{jk})_{r\times r} & (\lambda p_{jk})_{r\times n-r}\\
0_{n-r\times r} & Id_{n-r\times n-r}\end{pmatrix}$$
has an eigenvector $(x_1,\dots,x_n)^T$ with eigenvalue $1$ and $x_j>0$ for $j=1,\dots,n$.
\end{lemma}

\begin{proof} The equations defining a special graph weight for $\lambda\in (0,1)$ are
$$g(v_j)=\sum_{k=1}^n\lambda p_{jk}g(v_k)\ \ \ j=1,\dots,r,\qquad 
g(v_j)=\sum_{k=1}^n\delta_{jk}g(v_k)\ \ \ j=r+1,\dots,n.$$
This gives the necessary and sufficient condition for the existence of a special 
graph weight $\tilde{g}$ 
with $\tilde{g}(v_j)=x_j$. To get a state we normalise the eigenvector.
\end{proof}

The lemma can obviously be generalized to deal with general graph states on
finite graphs. Moreover we note that work in progress is extending the modular index 
theory to quasi-periodic actions of $\R$, and a modified version of the above lemma will give 
existence criteria in the quasi-periodic case also.

\begin{cor}\label{cr:useful-statement} 
Let $E$ be a locally finite directed graph without sources, and with finite 
zhyvot 
$M\subset E$. Let $(g,\lambda,n)$ be a special graph weight on $E$ for $\lambda\in(0,1)$,
$n|_{M^1}\equiv 1$ and $n|_{E^1\setminus M^1}\equiv 0$.
Then $\phi_{g,\lambda}$ extends to a
positive norm lower semi-continuous gauge invariant (usual gauge
action) functional on $C^*(E)$. The
functional $\phi_{g,\lambda}$ is faithful iff $(g,\lambda)$ is
faithful. We have the formula
$$ \phi_{g,\lambda}(ab)=\phi_{g,\lambda}(\s(b)a),\ \ \ a,b\in
A_c,$$
where 
$\s(S_\mu S_\nu^*)=\frac{\lambda(\nu)}{\lambda(\mu)}S_\mu S_\nu^*$ is a
densely defined regular automorphism of $C^*(E)$. In particular, $\phi_{g,\lambda}$
is a KMS weight on $C^*(E)$ for the (modified) zhyvot action
$$\s_t(S_\mu S_\nu^*)=\left(\frac{\lambda(\mu)}{\lambda(\nu)}\right)^{it}S_\mu S_\nu^*
=\lambda^{(|\mu|_\s-|\nu|_\s)it}S_\mu S_\nu^*.$$
\end{cor}

\begin{proof} 
The formula $\phi_{g,\lambda}(ab)=\phi_{g,\lambda}(\s(b)a)$ follows from 
Proposition \ref{pr:Hilbert-alg}. Together with the norm lower semicontinuity and the 
gauge invariance coming from Theorem \ref{thm:general-form}, 
we see that $\phi_{g,\lambda}$ is a KMS weight on $C^*(E)$.
\end{proof}

\subsection{The effect of field extensions}\label{fieldext2sec}

Suppose that we start with the infinite graph $\Delta_{\mathbb K}/\Gamma$ and
we pass to the graph $\Delta_{\mathbb L}/\Gamma$, for ${\mathbb L}$ a finite extension
of ${\mathbb K}$, by the procedure described in Section \ref{fieldextsect}. As
we have seen, this procedure consists of inserting $e_{{\mathbb L}/{\mathbb K}}-1$ new
vertices along edges and attaching infinite trees to the old and new
vertices, so that the resulting valence of all vertices is the
desired $q^f+1$.

Here we show that, if we have constructed a special graph weight for
$\Delta_{\mathbb K}/\Gamma$, then we obtain corresponding special graph
weights on all the $\Delta_{\mathbb L}/\Gamma$ for finite extensions ${\mathbb L}\supset
{\mathbb K}$. The special graph weight for $\Delta_{\mathbb L}/\Gamma$ is obtained from
that of $\Delta_{\mathbb K}/\Gamma$ by solving explicit equations.

\begin{prop} Let $E$ be a locally finite directed graph with no sources and with 
finite zhyvot $M$. Suppose that $(g,\lambda,n)$ is a faithful special graph weight on $E$
with $n|_M\equiv 1$ and $n|_{E\setminus M}\equiv 0$, and $\lambda\in (0,1)$. 
Let $F$ be the graph obtained from $E$ by inserting some new vertices along edges 
of $M$ and attaching any positive 
number of trees to the new vertices and any number of trees to the vertices
of $E$. Then $F$ has finite zhyvot $\tilde{M}$, with
$\tilde{M}^0=\{v\in F^0: v\in M^0\ \mbox{or}\ v=r(e),\ e\in F^1,\ s(e)\in M^0\}$ and
$\tilde{M}^1=\{e\in F^1:r(e)\in\tilde{M}^0\}$,
and a faithful special graph weight 
$(\tilde{g},\lambda,\tilde{n})$ with the same 
value of $\lambda$ and $\tilde{n}|_{\tilde{M}}\equiv 1$, 
$\tilde{n}|_{F\setminus\tilde{M}}\equiv 0$.
\end{prop}

\begin{proof} It is clear that $F$ is a graph and that $\tilde M$ is a zhyvot for $F$, 
since we can not introduce 
sources when vertices are only introduced splitting an existing edge into two, since
one of them has range the new vertex. 
Since extending a faithful graph state on the zhyvot $\tilde{M}$ to any graph obtained
by adding trees to vertices is possible, we need only be concerned with building a new 
special graph state on the zhyvot.

The problem turns out to be local, and we refer to Figure \ref{fig:insert-vertex} for 
the notation we shall use.

\begin{figure}[ht]
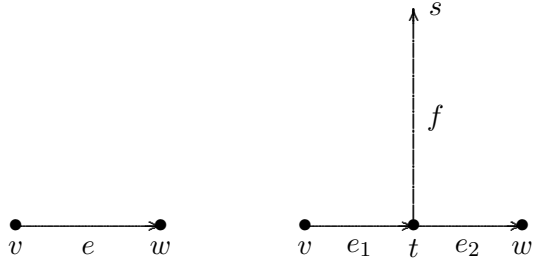

\[
\beginpicture
\setcoordinatesystem units <2.5em, 2.5em>
\setplotarea x from -2.5 to 9, y from -3.0 to 5.0
\put{$\bullet$} at -0.5 0
\put{$\bullet$} at 1.5 0
\put{$v$} at -0.5 -0.3
\put{$w$} at 1.5 -0.3
\setsolid
\plot -0.5 0 1.5 0 /
\arr(1.3,0)(1.5,0)
\put{$e$} at 0.5 -0.3
\plot 3.5 0 6.5 0 /
\arr(6.3,0)(6.5,0)
\put{$\bullet$} at 3.5 0
\put{$\bullet$} at 6.5 0
\put{$\bullet$} at 5.0 0
\put{$e_1$} at 4.25 -0.3
\put{$t$} at 5.0 -0.3
\put{$e_2$} at 5.75 -0.3
\put{$v$} at 3.5 -0.3
\put{$w$} at 6.5 -0.3
\plot 5.0 0 5.0 3 /
\arr(5.0,2.8)(5.0,3)
\arr(4.8,0.0)(5.0,0.0)
\put{$s$} at 5.3 3
\put{$f$} at 5.3 1.5
\endpicture
\]
\caption{Inserting a vertex}
\label{fig:insert-vertex}
\end{figure}

We suppose we have an edge $e$ with $s(e)=v$ and $r(e)=w$ in a graph
with special graph weight $(g,\lambda,n)$. So $g(v)=\lambda(e) g(w)+R$, where
$R=\sum_{s(f)=v,\ f\neq e}\lambda(e)g(r(f))$.

We now introduce a new vertex $t$ splitting $e$ into two edges
$e_1,e_2$ with $s(e_1)=v,\,r(e_1)=t,\,s(e_2)=t,\,r(e_2)=w$. We also
introduce a new edge $f$ with $s(f)=t$, $r(f)=s$ for some other vertex
$s$. We observe that we could add several edges $f_1,\dots,f_n$ 
with source $t$, and we indicate the modifications required in this case below.

We want to construct a special graph weight $\tilde{g}$ without changing our
parameter $\lambda$, or the values of the graph weight where it is
already defined. Thus we would like to solve

$$ \tilde{g}(v)=\lambda \tilde{g}(t)+R,\ \ \ 
\tilde{g}(t)=\lambda \tilde{g}(w)+\lambda \tilde{g}(s),\ \
\tilde{g}(v)=g(v),\ \ \tilde{g}(w)=g(w).$$

A solution to the above equations is as follows. Define $\tilde{g}=g$ on all
previously existing vertices, and on the vertex $s=r(f)$ set
$\tilde{g}(s)=\frac{1-\lambda}{\lambda}g(w)$. Then the above equations are
satisfied and we obtain $\tilde{g}(t)=g(w)$. If we have multiple edges $f_1,\dots,f_n$
then replacing $\tilde{g}(s)$ by $\sum_j\tilde{g}(r(f_j))$ we have a solution provided
$$ \sum_j\tilde{g}(r(f_j))=\frac{1-\lambda}{\lambda}g(w),$$
and thus we may just set the value of $\tilde{g}(r(f_j))$ to be 
$\frac{1}{n}\frac{1-\lambda}{\lambda}g(w)$. 
 
Finally, define $\tilde{n}$ by making it identically one on edges in $\tilde{M}$ and 
identically zero on other edges. Observe that $f$ is not an edge in the zhyvot.
\end{proof}

\section{Modular Index Invariants of Mumford Curves}\label{sec:mod-ind}

We have seen that we can associate directed graphs to Mumford
curves. These graphs consist of a finite graph along with trees
emanating out from some or all of its vertices. 
Though we do not have a general existence result, generically
we can construct ``special'' graph weights on such graphs.

{}From this we can construct both an equivariant Kasparov module
$(X,\D)$ and a
modular spectral triple $(\A,\HH,\D)$ as in \cite{CNNR}. Here the equivariance is with 
respect to the `modified zhyvot action' introduced in Corollary \ref{cr:useful-statement}. 
To compute the index pairing using the results of \cite{CNNR}, we need only be
able to compute traces of operators of the form $p\Phi_k$ where $p\in
F$ is a projection and the $\Phi_k$ are spectral projections of the
$\T$ action (or of $\D$ or of $\Delta$).

In the specific case of Mumford curves, the modular index pairings we would like to 
compute are with the
modular partial isometries arising from loops in the central graph corresponding to
the action on $\Delta_{\mathbb K}$ of each one of a chosen set of
generators $\{ \gamma_1, \ldots, \gamma_g \}$ of the Schottky group $\Gamma$.
These correspond to the fundamental closed geodesics in $\Delta_{\mathbb K}/\Gamma$,
by analogy to the fundamental closed geodesics in the hyperbolic 3-dimensional
handlebody $\H/\Gamma$ considered in \cite{Man}, \cite{CM}. The lengths of
these fundamental closed geodesics are the Schottky invariants of
$(\Gamma, \{ \gamma_1, \ldots, \gamma_g \})$ introduced above.

We introduce some notation so that we may effectively describe these
projections. The zhyvot of the graph we denote by $M$. 
Since outside of $M$ our graph is a union of trees, we may and do
suppose that the restriction of our graph weight to the exterior of
$M$ is a graph trace. That is, for all $v\notin M$ we have
$$ g(v)=\sum_{s(e)=v}g(r(e)),$$
and so for $e\notin M$, $\s_t(S_e)=S_e$.

In \cite{CNNR} we showed how to construct a Kasparov module for 
$A=C^*(E)$ and $F=A^\s$. We let $\Phi:A\to F$ be the expectation given
by averaging over the circle action, and define an inner product on $A$ 
with values in $F$ by setting
$$
(a|b):=\Phi(a^*b).
$$

We denote the $C^*$-module completion by $X$, and note that it is a 
full right $F$-module. There is an obvious action of $A$ by left multiplication,
and this action is adjointable. 

On the dense subspace $A_c\subset X$ we define an unbounded operator
$\D$ by defining it on generators and extending by linearity. We set
$$
\D S_\mu S_\nu^*:=(|\mu|_\s-|\nu|_\s)\,S_\mu S_\nu^*,
$$ so that up to a factor of $\log(\lambda)$, $\D$ is the generator of the 
zhyvot action. Observe that for a path $\mu$ contained in the 
exterior of $M$
we have $|\mu|_\s=0$. The closure of $\D$ is self-adjoint, regular, and 
for all $a\in A_c$ the endomorphism of $X$ given by $a(1+\D^2)^{-1}$ 
is a compact endomorphism. 

It is proved in \cite{CNNR} that $({}_AX_F,\D)$ is an equivariant Kasparov 
module for $A$-$F$ (with respect to the zhyvot action) and so it defines a 
class in $KK^{1,\T}(A,F)$.

Similarly, if we set $\HH:=\HH_{\phi_{g,\lambda}}$ to be the GNS space 
of $A$ associated to the weight $\phi_{g,\lambda}$, we obtain an unbounded
operator $\D$ (with the same definition on $A_c\subset \HH$). The 
triple $(\A,\HH,\D)$ is not quite a spectral triple. 

The compact endomorphisms of the $C^*$-module $X$, $End^0_F(X)$, 
act on $\HH$ in a natural fashion, \cite{CNNR}, and we define a von 
Neumann algebra by
$\cn=(End^0_F(X))''$. There is a natural trace 
$\mbox{Tr}_{\phi_{g,\lambda}}$ on $\cn$ satisfying 
$$
\mbox{Tr}_{\phi_{g,\lambda}}(\Theta_{x,y})=\phi_{g,\lambda}((y|x))
$$
for all $x,\,y\in X$. We define a weight $\phi_\D$ on $\cn$ by
$$
\phi_\D(T):=\mbox{Tr}_{\phi_{g,\lambda}}(\lambda^{\D}T),\ \ T\in\cn.
$$
Then the modular group of $\phi_\D$ is inner, and we let $\cM\subset\cn$
denote the fixed point algebra of the modular action. Then $\phi_\D$
restricts to a trace on $\cM$ and it is shown in \cite{CNNR} that
$$
f(1+\D^2)^{-1/2}\in\LL^{(1,\infty)}(\cM,\phi_\D),\ \ \ f\in F.
$$
Using this information it is shown in \cite{CNNR} that there is a pairing 
between $(A_c,\HH,\D)$ and homogenous (for the zhyvot action) 
partial isometries $v\in A_c$ with source and range projections in $F$.
The pairing is given by the spectral flow
$$
sf_{\phi_\D}(vv^*\D,v\D v^*)\in\R,
$$
this being well-defined since $v\D v^*\in\cM$. The numerical spectral
flow pairing and the equivariant $KK$ pairing are compatible.

In order to compute the spectral flow, we need explicit formulae for the 
spectral projections of $\D$ both as an operator on $X$ and as an 
operator on $\HH$. 

To this end, if $v\in E^0$ and $m>0$ we set $|v|_m=$the number of
paths $\mu$ with $|\mu|_\s=m$ and $r(\mu)=v$. It is important that our
graph is locally finite and has no sources so that $0<|v|_m<\infty$
for all $v\in E^0$ and $m>1$.

\begin{prop}\label{pr:spec-projs} The spectral projections of $\D$ can
be represented as follows:

$1)$ For $m>0$
$$ \Phi_m=\sum_{\substack{|\mu|_\s=m\\ s(\mu)\in M\\ r(\mu)\in M}}
\Theta_{S_\mu,S_\mu}.$$

$2)$ For $m=0$
$$\Phi_0=\sum_{v\in E^0}\Theta_{p_v,p_v}.$$

$3)$ For $m<0$, $v\in E^0$
$$ p_v\Phi_m=\frac{1}{|v|_{|m|}}\sum_{\substack{|\mu|_\s=|m|\\
r(\mu)=v}} \Theta_{S_\mu^*,S_\mu^*}.$$
In all cases, for (a subprojection of) a vertex projection $p_v$, the
operator $p_v\Phi_m$ is a finite rank endomorphism of the Kasparov
module $X$ and in the domain of $\phi_\D$ as an operator in
$\cM\subset\cn$.
\end{prop}

\begin{proof} We first recall that the $C^*$-module inner product is given by
$$(x|y)_R=\Phi(x^*y).$$
Now let $S_\rho S_\gamma^*\in X$ and with $|\mu|_\s>0$ consider
\begin{align*} \Theta_{S_\mu,S_\mu}S_\rho S_\gamma^*&=S_\mu(S_\mu|S_\rho
S_\gamma^*)_R\nno
&=\delta_{|\mu|_\s,|\rho|_\s-|\gamma|_\s}S_\mu S_\mu^*S_\rho S_\gamma^*\nno
&=\delta_{|\mu|_\s,|\rho|_\s-|\gamma|_\s}\delta_{\mu,\underline{\rho}}S_\rho
S_\gamma^*.
\end{align*}
Hence this is nonzero only when $|\rho|_\s=|\gamma|_\s+|\mu|_\s\geq |\mu|_\s$ and
$\underline{\rho}=\mu$. Thus when $|\rho|_\s-|\gamma|_\s=m>0$
$$\sum_{\substack{|\mu|_\s=m\\ s(\mu)\in M\\ r(\mu)\in M}}
\Theta_{S_\mu,S_\mu}S_\rho S_\gamma^*=
\Theta_{\underline{\rho},\underline{\rho}}S_\rho S_\gamma^*=S_\rho
S_\gamma^*,$$
and $\sum\Theta_{S_\mu,S_\mu}$  is zero on all other elements of
$X$. Hence the claim for the positive spectral projections is proved,
since finite sums of generators $S_\rho S_\gamma^*$ are dense in $X$. A
similar argument proves the claim for the zero spectral projection.

For the negative spectral projections, we observe that
\begin{align*} \Theta_{S_\mu^*,S_\mu^*}S_\rho,S_\gamma^*&=
S_\mu^*\delta_{r(\mu),s(\rho)}\delta_{|\mu|_\s+|\rho|_\s,|\gamma|_\s}S_\mu
S_\rho S_\gamma^*\nno
&=\delta_{r(\mu),s(\rho)}\delta_{|\mu|_\s+|\rho|_\s,|\gamma|_\s}S_\rho
S_\gamma^*.
\end{align*}
Summing over all paths $\mu$ with $|\mu|_\s=m>0$ and $r(\mu)=s(\rho)$
gives
$$\sum_{\substack{|\mu|_\s=m\\
r(\mu)=s(\rho)}}\Theta_{S_\mu^*,S_\mu^*}S_\rho S_\gamma^*=
\delta_{|\rho|_\s-|\gamma|_\s,-m}|s(\rho)|_mS_\rho S_\gamma^*.$$
Hence for a vertex $v\in E^0$
\begin{align*} \frac{1}{|v|_{|m|}}\sum_{\substack{|\mu|_\s=|m|\\
r(\mu)=v}} \Theta_{S_\mu^*,S_\mu^*}S_\rho
S_\gamma^*&=\delta_{|\rho|_\s-|\gamma|_\s,-m} \delta_{s(\rho),v}S_\rho
S_\gamma^*\nno
&=p_v\Phi_{-m}S_\rho S_\gamma^*.
\end{align*}
In all cases $p_v\Phi_k$ is a finite sum of rank one endomorphisms,
and so finite rank. In particular they are expressed as finite rank
endomorphisms of $\mbox{dom}(\phi)^{1/2}\subset X$, since for a graph
weight $g,\lambda$ all the $S_\mu$ and $S_\mu^*$ lie in the domain of the
associated weight $\phi$. This ensures that these endomorphisms extend
by continuity to the Hilbert space completion of
$\mbox{dom}(\phi)^{1/2}$ and by the construction of $\phi_\D$,
each $p_v\Phi_k\in\cM\subset\cn$ has finite {\em trace}. Similar
comments evidently apply to projections of the form $S_\mu S_\mu^*$
since this is a subprojection of $p_{s(\mu)}$.
\end{proof}

For large positive $k$, 
the computation of $\phi_\D(S_\mu S_\mu^*\Phi_k)$ is extremely
difficult, and needs to be handled on a `graph-by-graph'
basis. However it turns out that we need only compute for $|k|\leq |\mu|_\s$, and 
this is completely tractable.

\begin{lemma}\label{lem:trace-comp} Let $\gamma$ be a path in $E$ with
$s(\gamma),\,r(\gamma)\in M$ and $|\gamma|_\s>0$. Then for all $k\in\Z$ with
$|\gamma|_\s\geq |k|$ we have
$$\phi_\D(S_\gamma S_\gamma^*\Phi_k)=\phi_{g,\lambda}(S_\gamma
S_\gamma^*)=\lambda^{|\gamma|_\s}g(r(\gamma)).$$
For a path of length zero (i.e. a vertex $v$) in $M$ and $k<0$ we have
$$ \phi_\D(p_v\Phi_k)=\phi_{g,\lambda}(p_v)=g(v).$$
\end{lemma}

\begin{proof} We begin with $|\gamma|_\s\geq k>0$. In this case the
definitions yield
\begin{align*} \phi_\D(S_\gamma S_\gamma^*\Phi_k)&=
\sum_{\substack{|\mu|_\s=k\\ s(\mu)\in M\\ r(\mu)\in
M}}\phi_\D(S_\gamma S_\gamma^*\Theta_{S_\mu,S_\mu})\nno
&=\sum_{\substack{|\mu|_\s=k\\ s(\mu)\in M\\ r(\mu)\in
M}}\lambda^{k}\phi_{g,\lambda}(S_\mu^*S_\gamma S_\gamma^*S_\mu)\nno
&=\lambda^{k}\phi_{g,\lambda}(S_{\overline{\gamma}}S^*_{\overline{\gamma}})\nno
&=\lambda^{|\gamma|_\s}
\phi_{g,\lambda}(S^*_{\overline{\gamma}}S_{\overline{\gamma}})=
\lambda^{|\gamma|_\s}\phi_{g,\lambda}(p_{r(\gamma)})\nno 
 &=\phi_{g,\lambda}(S_\gamma S_\gamma^*).
\end{align*}
So now consider $|\gamma|_\s>0$ or $\gamma=v$ for some vertex $v\in M$
and $k>0$. In the latter case, $S_\gamma S_\gamma^*=p_vp_v=p_v=S_\gamma^*
S_\gamma$. Then
\begin{align*} \phi_\D(S_\gamma S_\gamma^*\Phi_{-k})
&=\frac{1}{|s(\gamma)|_k} \sum_{\substack{|\mu|_\s=k\\
r(\mu)=s(\gamma)}}\lambda^{-k} \phi_{g,\lambda}(S_\mu S_\gamma S_\gamma^*
S_\mu^*)\nno
&=\frac{1}{|s(\gamma)|_k} \sum_{\substack{|\mu|_\s=k\\
r(\mu)=s(\gamma)}}\lambda^{|\gamma|_\s}\phi_{g,\lambda}(p_{r(\gamma)})\nno
&=\lambda^{|\gamma|_\s}\phi_{g,\lambda}(p_{r(\gamma)})\nno
&=\phi_{g,\lambda}(S_\gamma S_\gamma^*).
\end{align*}
This completes the proof.
\end{proof}

We now have the necessary ingredients to compute the modular index
pairing with $S_\gamma$ where $\gamma$ here denotes a loop contained in 
the finite graph $M=\Delta'_\Gamma/\Gamma$ and corresponding
to an element in the chosen set of generators of the Schottky group. 
We suppose that $k=|\gamma|_\s$ is non-zero, so that the loop is non-trivial, 
and denote $\beta:=-\log\lambda$.

By \cite[Lemma 4.10]{CNNR} and Lemma \ref{lem:trace-comp} we have
\begin{align*}
sf_{\phi_\D}(S_\gamma S_\gamma^*\D,S_\gamma\D S_\gamma^*)
&=-\sum_{j=0}^{k-1}e^{-\beta j}\mbox{Tr}_\phi(S_\gamma S_\gamma^*\Phi_j)\\
&=-\sum_{j=0}^{k-1}\phi_\D(S_\gamma S_\gamma^*\Phi_j)\\
&=-k\phi_{g,\lambda}(S_\gamma S_\gamma^*)\\
&=-k\lambda^{k}g(r(\gamma))=-ke^{-\beta k}g(r(\gamma)).
\end{align*}
Since we assume that $(g,\lambda)$ are given as part of the data of our special graph weight,
we can extract the integer $k$. Moreover, $k$ determines the value of the index pairing.

Thus, we see that the Schottky invariants of the data 
$(\Gamma, \{\gamma_1,\ldots,\gamma_g\})$
can be recovered from the modular index pairing and in fact determine it, for a given graph weight $(g,\lambda)$. This confirms the fact that the noncommutative geometry of the graph algebra
$C^*(E)=C^*(\Delta_{\mathbb K}/\Gamma)$ maintains the geometric information related to the action of
the Schottky group on the Bruhat--Tits tree $\Delta_{\mathbb K}$. This is still less information than being able to
reconstruct the curve, since the Schottky invariants only depend on the valuation. We show explicitly in the next section how the construction of graph weights works in some simple examples of Mumford curves.

\section{Low genus examples}

We consider here the cases of the elliptic curve with Tate 
uniformization (genus one case)
and the three genus two cases considered in \cite{CM1}.  In each 
of these examples we
give an explicit construction of graph weights and compute the 
relevant modular index pairings, 
showing that one recovers from them the Schottky invariants. 
Notice that, for the genus two cases,
the finite zhyvot graphs $\Delta_\Gamma/\Gamma$ are the 
same considered in \cite{CM1}, 
which we report here,
though in the present setting we work with the infinite graphs 
$\Delta_{\mathbb K}/\Gamma$. 
We discuss here the graph weights equation on the 
zhyvot graph and on the infinite graph  $\Delta_{\mathbb K}/\Gamma$.

{\bf Example: Genus 1} As a first application to Mumford curves we
consider the simplest case of genus one. In this case, the Schottky
uniformization is the Tate uniformization of $p$-adic elliptic
curves. The $p$-adic Schottky group is just a copy of $\Z$ generated
by a single hyperbolic element in $\PGL_2({\mathbb K})$. In this case the
graph $\Delta_{\mathbb K}/\Gamma$ will be always of the form illustrated in
Figure \ref{Mumford}, with a central polygon with $n$ vertices 
and trees departing from its vertices. With our convention on the
orientations, the edges are oriented in such a way as to go around the
central polygon, while the rest of the graph, \ie the trees stemming
from the vertices of the polygon, are oriented away from it and
towards the boundary $X_\Gamma =\partial \Delta_{\mathbb K}/\Gamma$.

\begin{figure}
\begin{center}
\includegraphics[scale=0.8]{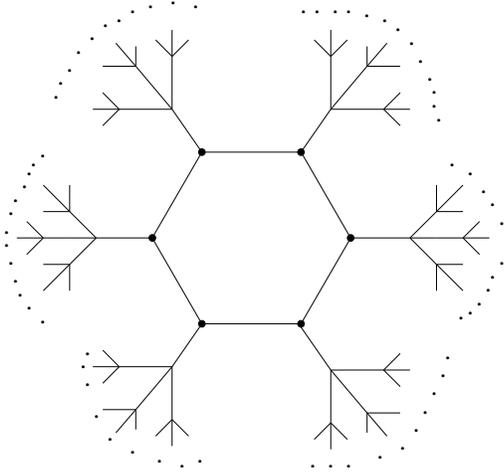}
\caption{The genus one case.
\label{Mumford}}
\end{center}
\end{figure}

Label the vertices on the polygon by $v_i$, $i=1,\dots,n$. To get a special
graph weight we need $0<\lambda<1$ and a function $g$ on the vertices such that
$$ g(v_i)=\lambda g(v_{i+1})+B_i,\qquad B_i=\sum_{v_{i+1}\neq w=r(e),\ s(e)=v_i}g(w).$$
To simplify we suppose all the $g(v_i)$ are equal, $\sum g(v_i)=1$ and all the $B_i$ 
are equal. 
Then we obtain a special graph
weight for any $\lambda<1$ by setting
$$ g(v_i)=\frac{1}{n},\qquad
B_i=\frac{1-\lambda}{n}.$$
For each $i$ we can now define the various $g(w)$ appearing in the sum 
defining $B_i$ by $g(w)=\frac{1}{m_i}g(w)$ where $m_i$ is the number of such $g(w)$. 
This graph weight can be extended to the rest of the trees as a graph
trace, and the associated $\T$ action is nontrivial on each $S_e$,
$e\in M^1$ where $M$ is just the central polygon. Hence choosing $\gamma$ 
to be the (directed) path which goes
once around the polygon (the choice of $r(\gamma)=s(\gamma)$ is
irrelevant) gives
$$\la [\gamma],\phi_{g,\lambda}\ra=-\lambda^{n}.$$

{\bf Example: Genus 2} 
In the case of genus two, the possible graphs $\Delta_\Gamma/\Gamma$
and the corresponding special fibers of the algebraic curve are 
illustrated in Figure \ref{graphs}, which we reproduce from \cite{CM1}, 
see also \cite{Mar}.

\begin{figure}
\begin{center}
\epsfig{file=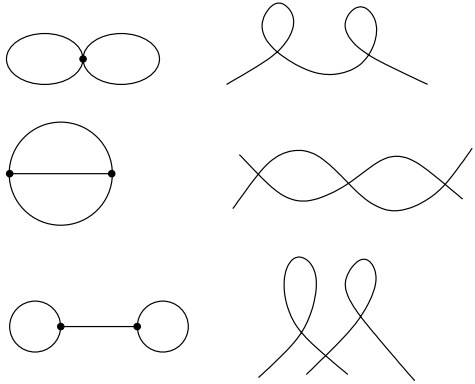} \caption{The graphs
$\Delta_\Gamma/\Gamma$ for genus $g=2$, and the corresponding
fibers. \label{graphs}}
\end{center}
\end{figure}

We see more in detail the various cases. These are the same cases considered in \cite{CM1}.

{\bf Case 1}: In the first case, the tree $\Delta_\Gamma$ is just
a copy of the Cayley graph of the free group $\Gamma$ on two
generators as in Figure \ref{trees1}.

\begin{figure}
\begin{center}
\includegraphics[scale=0.55]{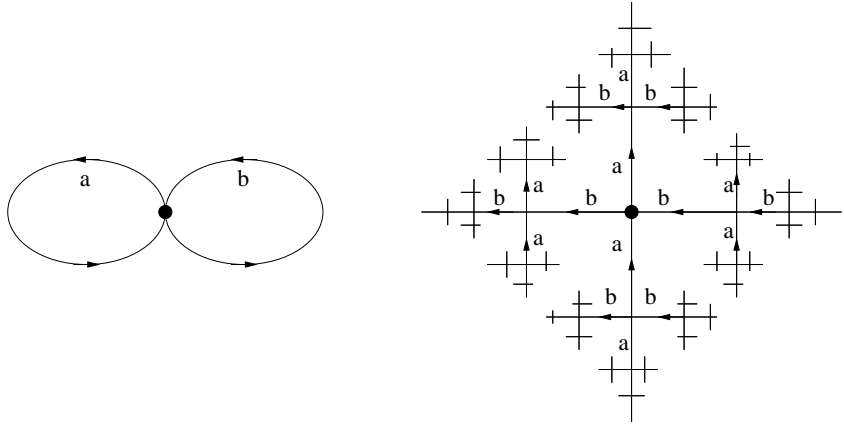}
\caption{Genus two: first case 
\label{trees1}}
\end{center}
\end{figure}

The graph algebra of this graph is the Cuntz algebra $O_2$. The only possible 
special graph state is $g(v)=1$ for the single vertex and $\lambda=1/2$. This corresponds
to the usual gauge action and its unique KMS state, \cite{CPR2}. Once
we add trees to this example, many more possibilities for the KMS 
weights appear.

{\bf Case 2}: In the second case, the finite directed graph
$\Delta_\Gamma/\Gamma$ is of the form illustrated in Figure
\ref{trees2}. We label by $a=e_1$,
$b=e_2$ and $c=e_3$ the oriented edges in the graph
$\Delta_\Gamma/\Gamma$, so that we have a corresponding set of
labels $E=\{ a,b,c,\bar a, \bar b, \bar c \}$ for the edges in the
covering $\Delta_\Gamma$. A choice of generators for the group
$\Gamma \simeq \Z * \Z$ acting on $\Delta_\Gamma$ is obtained by
identifying the generators $\gamma_1$ and $\gamma_2$ of $\Gamma$ with the
chains of edges $a b$ and $a \bar c$, hence the orientation on
the tree $\Delta_\Gamma$ and on the quotient graph is as
illustrated in the figure.

\begin{figure}
\begin{center}
\includegraphics[scale=0.55]{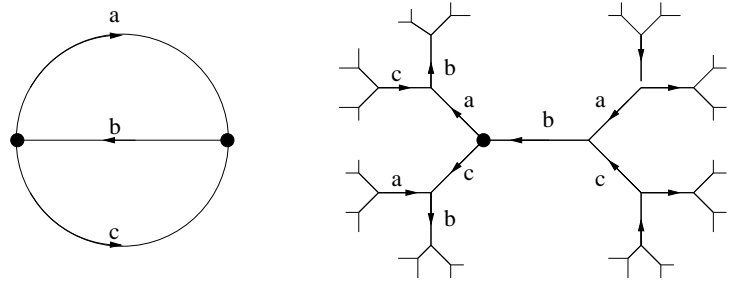}
\caption{Genus two: second case 
\label{trees2}}
\end{center}
\end{figure}

There are four special graph states on this graph algebra (up to swapping the roles of 
the edges $a$ and $c$). Let $v=s(b)$ and $w=r(b)$. Let $n_1=n(b)$, $n_2=n(a)$
and $n_3=n(c)$ with each $n_j\in\{0,1\}$.
Then the various states are described in Table \ref{t:states}. To fit with our 
requirement that the zhyvot action `sees' every edge in the zhyvot,
we should adopt only the last choice of state. Of course, we are again 
neglecting the trees, and including them would give us many more options.

\begin{table}
\caption{graph states for Case 2}
\begin{center}
\begin{tabular}{|c|c|c|c|c|c|}
\hline
$n_1$ & $n_2$ & $n_3$ & $\lambda$ & $g(v)$ & $g(w)$\\
\hline
$0$ & $0$ & $0$\mbox{ or }$1$ & --- & --- & ---\\
\hline
$0$ & $1$ & $1$ & $\frac{1}{2}$ & $\frac{1}{2}$ & $\frac{1}{2}$\\
\hline
$1$ & $0$ & $0$ & $\frac{1}{2}$ & $\frac{1}{3}$ & $\frac{2}{3}$\\
\hline
$1$ & $0$ & $1$ & $\frac{-1+\sqrt{5}}{2}$ & $\lambda^2$ & $\lambda$\\
\hline
$1$ & $1$ & $1$ & $\frac{1}{\sqrt{2}}$ & $\frac{1}{1+\sqrt{2}}$ & 
$\frac{\sqrt{2}}{1+\sqrt{2}}$ \\
\hline
\end{tabular}
\end{center}
\label{t:states}
\end{table}%

{\bf Case 3}: In the third case the obtained
oriented graph is the same as the graph of $SU_q(2)$ of
Figure \ref{SUq2Fig}. We have already described the graph states 
for $SU_q(2)$. In this case, the inclusion of trees is necessary to obtain 
a special graph weight adapted to the zhyvot action.
In fact, a choice of
generators for the group $\Gamma\simeq \Z * \Z$ acting on
$\Delta_\Gamma$ is given by $ab\bar a$ and $c$, so that the obtained
orientation is as in the figure.

\begin{figure}
\begin{center}
\includegraphics[scale=0.55]{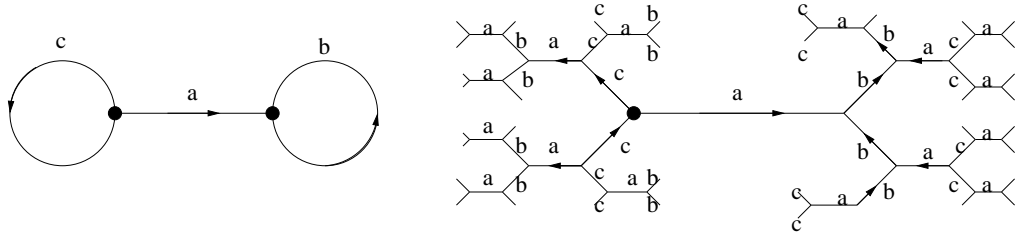}
\caption{Genus two: third case 
\label{trees3}}
\end{center}
\end{figure}

When we considers the tree $\Delta_\Gamma'$ instead of
$\Delta_\Gamma$ one is typically adding extra vertices. The way the
tree $\Delta_\Gamma'$ sits inside the Bruhat--Tits tree $\Delta_{\mathbb K}$ 
and in particular how many extra vertices of $\Delta_{\mathbb K}$ are present 
on the graph $\Delta_\Gamma'/\Gamma$ with respect to the vertices of
$\Delta_\Gamma/\Gamma$ gives some information on the uniformization,
\ie it depends on where the Schottky group $\Gamma$ lies in
$\PGL_2({\mathbb K})$, unlike the information on the graph
$\Delta_\Gamma/\Gamma$ which is purely combinatorial. For example, in
the genus two case of Figure \ref{graphs} one can have graphs 
$\Delta_\Gamma'/\Gamma$ and $\Delta_{\mathbb K}/\Gamma$ of the form as in Figure
\ref{graphsK}.

\begin{center}
\begin{figure}
\includegraphics[scale=0.6]{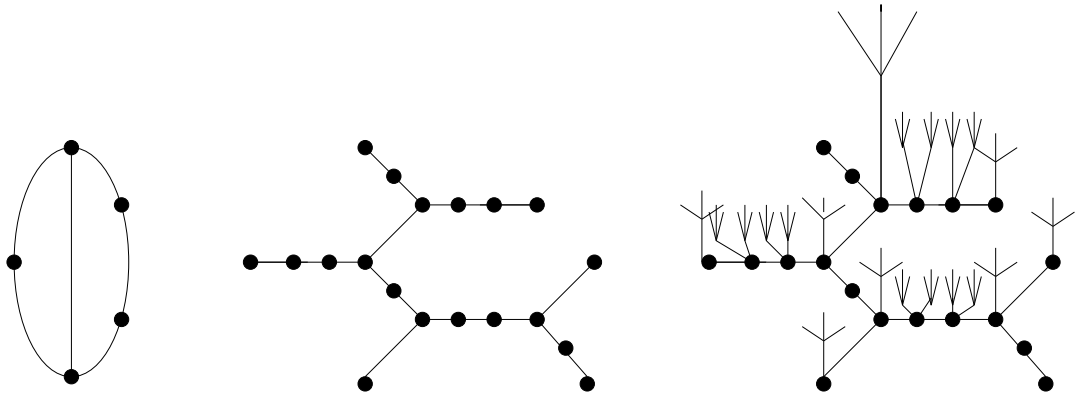}
\caption{An example of a graph $\Delta_\Gamma'/\Gamma$, the tree
$\Delta_\Gamma'$ and its embedding in $\Delta_{\mathbb K}$ for ${\mathbb K}=\Q_3$.
\label{graphsK}}
\end{figure}
\end{center}

{\bf Remark: Higher genus} In the higher genus case one knows by
\cite[p.124]{GvP}  that any stable graph can occur as the graph
$\Delta_\Gamma/\Gamma$ of a Mumford curve. By stable graph we mean a 
finite graph which is connected and such that each vertex that is not
connected to itself by an edge is the source of at least three edges. 

Thus, the combinatorial complexity of the graph is pretty much
arbitrary. One can also assume, possibly after passing to a finite
extension of the field ${\mathbb K}$, that there are infinite homogeneous trees
attached to each vertex of the graph $\Delta_\Gamma/\Gamma$.

We make the final remark that the restriction to circle actions in this paper
is likely an artificial one. In work in progress, KMS index theory is extended to 
quasi-periodic actions of the reals. As the action associated to {\em any} graph
weight will be quasi-periodic, this will hopefully allow us to prove general existence 
theorems for (quasi-periodic) graph weights compatible with the zhyvot action.

\section{Jacobian and theta functions}

We recall here briefly the relation between the Jacobian and theta
functions of a Mumford curve and the group of currents on the infinite
graph $\Delta_{\mathbb K}/\Gamma$.

Recall first that a current on a locally finite graph $G$ is an 
integer valued function of the oriented edges of $G$ that 
satisfies the following properties.
\begin{enumerate}
\item Orientation reversal:
\begin{equation}\label{curr1}
\mu(\bar e)=-\mu(e),
\end{equation}
where $\bar e$ is the edge $e$ with the reverse orientation.
\item Momentum conservation:
\begin{equation}\label{curr2}
\sum_{s(e)=v}\mu(e)=0.
\end{equation}
\end{enumerate}
One denotes by $\cC(G)$ the abelian group of currents on the graph $G$. 

Suppose we are given a tree $\cT$. Then the group $\cC(\cT)$ can be 
equivalently described as the group of finitely additive measures of 
total mass zero on the space $\partial \cT$ of ends of the tree by 
setting $m(U(e))=\mu(e)$, where $U(e)\subset \partial \cT$ is the 
clopen set of ends of the infinite half lines starting at the vertex 
$s(e)$ along the direction $e$. For $G=\cT/\Gamma$, 
the group of currents $\cC(G)=\cC(\cT)^\Gamma$ can be identified with
the group of $\Gamma$-invariant measures on $\partial \cT$, 
\ie finitely additive measures of total mass zero on 
$\partial \cT/\Gamma$.

As above, we let $X=X_\Gamma$ be a Mumford curve, uniformized 
by the $p$-adic Schottky group $\Gamma$. We consider the above applied to
the tree $\Delta_{\mathbb K}$ with the action of the Schottky group $\Gamma$ and 
the infinite, locally finite quotient graph $\Delta_{\mathbb K}/\Gamma$ with 
$\partial\Delta_{\mathbb K}/\Gamma =X_\Gamma ({\mathbb K})$. 

It is known (see \cite{vdP}, Lemma 6.3 and Theorem 6.4) 
that the Jacobian of a Mumford curve can 
be described, as an analytic variety, via the isomorphism
\begin{equation}\label{PicX}
{\rm Pic}^0(X) \cong \Hom(\Gamma,{\mathbb K}^*)/c(\Gamma_{ab}),
\end{equation}
where $\Gamma_{ab}=\Gamma/[\Gamma,\Gamma]$ denotes the abelianization, 
$\Gamma_{ab}\cong \Z^g$, with $g$ the genus, and the homomorphism
\begin{equation}\label{homc}
c: \Gamma_{ab} \to \Hom(\Gamma_{ab},{\mathbb K}^*)
\end{equation}
is defined by the first map in the homology exact sequence
\begin{equation}\label{Hseq}
0 \to \cC(\Delta_{\mathbb K})^\Gamma \stackrel{c}{\to} \Hom(\Gamma, {\mathbb K}^*)\to 
H^1(\Gamma, \cO(\Omega_\Gamma)^*) \to H^1(\Gamma, \cC(\Delta_{\mathbb K})) \to 0, 
\end{equation}
associated to the short exact sequence
\begin{equation}\label{Oseq}
0 \to {\mathbb K}^* \to \cO(\Omega_\Gamma)^* \to \cC(\Delta_{\mathbb K}) \to 0
\end{equation}
of Theorem 2.1 of \cite{vdP}, where $\cO(\Omega_\Gamma)^*$ is the 
group of invertible holomorphic functions on $\Omega_\Gamma\subset \P^1$. 

In the sequence \eqref{Hseq}, one uses the fact that 
$H^i(\Gamma)=0$ for $i\geq 2$ and the identification
\begin{equation}\label{currGammaAB}
 \cC(\Delta_{\mathbb K})^\Gamma =H^0(\Gamma,\cC(\Delta_{\mathbb K}))=\Gamma_{ab}
=\pi_1(\Delta_{\mathbb K}/\Gamma)_{ab}, 
\end{equation}
see \cite{vdP}, Lemma 6.1 and Lemma 6.3. One can then use the 
short exact sequence
\begin{equation}\label{shortexact}
0 \to \cC(\Delta_{\mathbb K})\to \cA(\Delta_{\mathbb K})\stackrel{d}{\to} \cH(\Delta_{\mathbb K})\to 0, 
\end{equation}
where $\cC(\Delta_{\mathbb K})$ is the group of currents on the 
Bruhat--Tits tree $\Delta_{\mathbb K}$, $\cA(\Delta_{\mathbb K})$ is the group 
of integer values functions on the set of edges of $\Delta_{\mathbb K}$ 
satisfying $h(\bar e)=-h(e)$ under orientation reversal 
$e\mapsto \bar e$ and $\cH(\Delta_{\mathbb K})$ is the group of integer 
valued functions on the set of vertices of $\Delta_{\mathbb K}$. The map 
$d$ in \eqref{shortexact} is given by 
\begin{equation}\label{dmap}
d:\cA(\Delta_{\mathbb K})\to \cH(\Delta_{\mathbb K}), \ \ \ \ d(h)(v)=\sum_{s(e)=v} h(e).
\end{equation}
The long exact homology sequence associated to \eqref{shortexact} is given by 
\begin{equation}\label{longexact}
0 \to \cC(\Delta_{\mathbb K}/\Gamma)\to \cA(\Delta_{\mathbb K}/\Gamma)
\stackrel{d}{\to} \cH(\Delta_{\mathbb K}/\Gamma)\stackrel{\Phi}{\to} 
H^1(\Gamma,\cC(\Delta_{\mathbb K}))\to 0, 
\end{equation}
where one has $H^1(\Gamma,\cC(\Delta_{\mathbb K}))\cong \Z$ and, under this 
identification, the last map in the exact sequence is given by
\begin{equation}\label{mapPhi}
 \Phi:\cH(\Delta_{\mathbb K}/\Gamma)\to 
H^1(\Gamma,\cC(\Delta_{\mathbb K}))=\Z, \ \ \  
\Phi(f)=\sum_{v\in (\Delta_{\mathbb K}/\Gamma)^0} f(v). 
\end{equation}
Moreover, one has an identification
$$ H^1(\Gamma, \cO(\Omega_\Gamma)^*)=H^1(X,\cO_X^*) ={\rm Pic}(X), $$
the group of equivalence classes of holomorphic (hence by GAGA algebraic) 
line bundles on the curve $X$, and the last map in the exact 
sequence \eqref{Hseq} is then given by the degree map 
$\deg: {\rm Pic}(X) \to \Z$, whose kernel is the Jacobian 
$J(X)={\rm Pic}^0(X)$, see \cite{vdP} Lemma 6.3. 

A theta function for the Mumford curve $X=X_\Gamma$ is an invertible 
holomorphic function $f\in \cO(\Omega_\Gamma)^*$ such that
$$ \gamma^* f = c(\gamma) f, \ \ \ \forall \gamma \in \Gamma, $$
with $c\in \Hom(\Gamma,{\mathbb K}^*)$ the automorphic factor.
The group $\Theta(\Gamma)$ of theta functions of the curve $X$ is then 
obtained from the exact sequences \eqref{Oseq} and \eqref{Hseq} as
(\cite{vdP}) 
\begin{equation}\label{ThetaX}
0 \to {\mathbb K}^* \to \Theta(\Gamma) \to \cC(\Delta_{\mathbb K})^\Gamma \to 0.
\end{equation}

More precisely, let $\H_{\mathbb K} =\P^1_{\mathbb K} \smallsetminus \P^1({\mathbb K})$ be
Drinfeld's p-adic upper half plane. It is well known (see for instance
the detailed discussion given in \cite{BouCar} \S I.1 and \S I.2)
that $\H_{\mathbb K}$ is a rigid analytic space endowed with a surjective map 
\begin{equation}\label{HandDelta}
\Lambda: \H_{\mathbb K} \to \Delta_{\mathbb K}
\end{equation}
to the Bruhat--Tits tree $\Delta_{\mathbb K}$ such that, for vertices $v,w\in
\Delta_{\mathbb K}^0$ with $v=s(e)$ and $w=r(e)$, for an edge $e\in \Delta_{\mathbb K}^1$,
the preimages $\Lambda^{-1}(v)$ and $\Lambda^{-1}(w)$ are open subsets of
$\Lambda^{-1}(e)$. The picture of
the relation between $\H_{\mathbb K}$ and $\Delta_{\mathbb K}$ through the map $\Lambda$
is given in Figure \ref{DriFig}.

Given a theta function $f\in \Theta(\Gamma)$, the associated current
$\mu_f\in \cC(\Delta_{\mathbb K})^\Gamma$ obtained as in \eqref{ThetaX} is given
explicitly by the growth of the spectral norm in the Drinfeld upper
half plane when moving along an edge in the Bruhat--Tits tree, that is, 
\begin{equation}\label{mufnorm}
\mu(e)=  \log_q \| f\|_{\Lambda^{-1}(r(e))} - \log_q \| f
\|_{\Lambda^{-1}(s(e))},
\end{equation}
with $q=\#\cO/\m$ and $\|f\|_{\Lambda^{-1}(v)}$ is the spectral norm
$$ \|f\|_{\Lambda^{-1}(v)}=\sup_{z\in \Lambda^{-1}(v)} |f(z)| $$
with $|\cdot|$ the absolute value with $|\pi|=q^{-1}$, with $\pi$ a uniformizer,
that is, $\m=(\pi)$.

The case of function fields over a finite field ${\mathbb F}_q$ of
characteristic $p$ is similar to the $p$-adic case, with 
$\H_{\mathbb K}$ and $\Delta_{\mathbb K}$ the Drinfeld upper half plane and the
Bruhat--Tits tree in characteristic $p$ (see for instance
\cite{GekW}). 

\begin{figure}
\begin{center}
\includegraphics[scale=0.7]{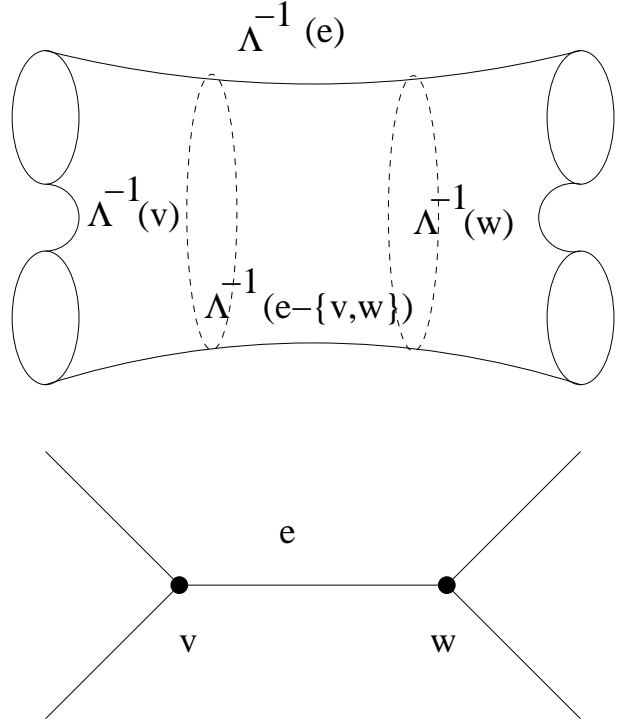}
\caption{The $p$-adic upper half-plane and
the Bruhat--Tits tree
\label{DriFig}}
\end{center}
\end{figure}

\smallskip

\subsection{Graph weights, currents, and theta functions}\label{weightthetasec}

We now show how to relate theta functions on the Mumford curve to
graph weights. The type of graph weights we consider here will in general not
be special graph weights as those we considered in the previous sections.
In fact, we will see that, when constructing currents from graph weights, we
need to work with functions $\lambda(e)$ of the special form $\lambda(e)=N_{\bar e}^{-1}$,
where $N_e$ is defined as in \eqref{Ne} below. Along the outer trees of the
graph $\Delta_{\mathbb K}/\Gamma$, the fact that these are homogeneous trees of
valence $q+1$ (or $q^f +1$ for field extensions) and the orientation of these trees
away from the zhyvot graph $\Delta_\Gamma '/\Gamma$ gives that $\lambda(e)=\lambda^{n_e}$
for $\lambda =q^{-1} \in (0,1)$ (or $\lambda = q^{-f}$ for field extensions) and $n_e=1$, the
expression for $N_e$ inside $\Delta_\Gamma '/\Gamma$ depends on the orientation on
$\Delta_\Gamma '$ described in Lemma \ref{defor}. Similarly, as we see below,
when we construct
(inhomogeneous or signed) graph weights from currents, we use the function
$\lambda(e) = N_e^{-1}$, which is again of the form $q^{-1}$ (or $q^{-f}$) on the 
outer trees of $\Delta_{\mathbb K}/\Gamma$, but which also depends on the given 
orientation inside $\Delta_\Gamma'/\Gamma$. It will be interesting to consider the
quasi-periodic actions associated to these types of graph weights. 

We first show that the same methods that produce graph weights
can be used to construct real valued 
currents on the graph $\Delta_{\mathbb K}/\Gamma$. 

\begin{lemma}\label{weightcurr}
Let $(g,\lambda)$ be a graph weight on the infinite, locally finite graph
$\Delta_{\mathbb K}/\Gamma$. For an oriented edge 
$e\in (\Delta_{\mathbb K}/\Gamma)^1$ let
\begin{equation}\label{Ne}
N_e:=\#\{ e'\in (\Delta_{\mathbb K}/\Gamma)^1\,|\, s(e')=s(e)\}.
\end{equation}
Then the function
\begin{equation}\label{momentumweight}
\mu(e):= \lambda(e) g(r(e))- \frac{1}{N_e} g(s(e))
\end{equation}
satisfies the momentum conservation equation \eqref{curr2}. Moreover, 
if $g:(\Delta_{\mathbb K}/\Gamma)^0 \to [0,\infty)$ is a function on
the vertices of the graph such that $(g,\lambda)$ is a graph weight
for $\lambda(e)=N_{\bar e}^{-1}$, then the function $\mu:
(\Delta_{\mathbb K}/\Gamma)^1\to \R$ given by \eqref{momentumweight}
is a  real valued current on $\Delta_{\mathbb K}/\Gamma$.
\end{lemma}

\proof The first result is a direct consequence of the equation for graph
weights 
\begin{equation}\label{weighteq}
g(v)=\sum_{s(e)=v} \lambda(e) g(r(v)) ,
\end{equation}
which gives the equation \eqref{curr2} for \eqref{momentumweight}.
The second statement also follows from \eqref{weighteq}, where in this case
the resulting function $\mu:
(\Delta_{\mathbb K}/\Gamma)^1\to \R$ given by \eqref{momentumweight}
also satisfies the orientation-reversed equation
\begin{equation}\label{corr2t}
\sum_{r(e)=v} \mu(e) =0.
\end{equation}
In particular, it also satisfies 
$$ \mu(\bar e)=\frac{1}{N_e}  g(s(e))- \frac{1}{N_{\bar e}} g(r(e))
=-\mu(e), $$ hence it 
defines a real valued current on $\Delta_{\mathbb K}/\Gamma$, 
$$ \mu\in \cC(\Delta_{\mathbb K}/\Gamma)\otimes_\Z \R. $$
\endproof

Conversely, one can use theta functions on the Mumford curve to
construct graph weights on the tree, which however do not have
the positivity property. We introduce the following notions
generalizing that of positive graph weight given in Definition
\ref{grapwdef}. 

\begin{defn}\label{virtgrweight}
Given a graph $E$, we define an inhomogeneous graph weight to be a triple
$(g,\lambda,\chi)$ of non-negative functions $g:E^0 \to \R_+=[0,\infty)$,
$\lambda, \chi: E^1\to [0,\infty)$ satisfying
\begin{equation}\label{inhomweight}
g(v)+d\chi(v)=\sum_{s(e)=v} \lambda(e) g(r(e)),
\end{equation}
where, as above, $d\chi(v)=\sum_{s(e)=v} \chi(e)$.
A rational virtual graph weight is a pair $(g,\lambda)$ of functions 
$g:E^0 \to \Q$  and $\lambda:E^1\to \Q_+=\Q\cap [0,\infty)$ such that
there exist rational valued inhomogeneous graph weights
$(g^\pm,\lambda,\chi)$ with $g(v)=g^+(v)-g^-(v)$ for all $v\in E^1$.
\end{defn}

A virtual graph weight satisfies the equation \eqref{weighteq}.
The choice of the inhomogeneous weights $(g^\pm,\lambda,\chi)$
that give the decomposition $g(v)=g^+(v)-g^-(v)$ is non-unique.

\begin{lemma}\label{thetaweights}
Let $f\in \Theta(\Gamma)$ be a theta function for the Mumford curve 
$X=X_\Gamma$, with $\Theta(\Gamma)$ as in \eqref{ThetaX}. 
Then $f$ defines an associated pair of rational valued functions
$(g,\lambda)$ on the tree $\Delta_{\mathbb K}$, with $\lambda(e)=N_e^{-1}$ 
and with $g: \Delta_{\mathbb K}^0 \to \Q$ satisfying the equation 
\eqref{weighteq}.
\end{lemma}

\proof By \eqref{ThetaX} we know that the theta function $f\in
\Theta(\Gamma)$ determines an associated integer valued current 
$\mu_f$ on the graph $\Delta_{\mathbb K}/\Gamma$, that is, an element in
$\cC(\Delta_{\mathbb K})^\Gamma=\cC(\Delta_{\mathbb K}/\Gamma)$. 
We view $\mu_f$ as a current on the tree $\Delta_{\mathbb K}$ that is
$\Gamma$-invariant. We also know that the current $\mu=\mu_f$ 
is given by the expression \eqref{mufnorm} in terms of the 
spectral norm on the Drinfeld $p$-adic upper half plane. 

If we set 
\begin{equation}\label{gspnormf}
g(v):= \log_q \| f \|_{\Lambda^{-1}(v)},
\end{equation}
We see easily that the condition \eqref{curr2} for the current
$\mu(e)=g(r(e))-g(s(e))$ implies that the function $g: \Delta_{\mathbb K}^0 \to
\Z$ satisfies the weight equation \eqref{weighteq} with
$\lambda(e)=N_e^{-1}$. In fact, we have
$$ \sum_{s(e)=v} g(r(e)) = N_v g(v), $$
with $N_v=\#\{ e': s(e')=v\}=N_e$, for all $e$ with $s(e)=v$.
\endproof

\begin{lemma}\label{lemvirtg}
The function $g: \Delta_{\mathbb K}^0\to \Z$ associated to a theta function in
$f\in \Theta(\Gamma)$ is an integer valued rational virtual graph weight.
\end{lemma}

\proof The measure $\mu=\mu_f$ can be written
(non-uniquely) as a difference 
\begin{equation}\label{mudiff}
\mu(e) = \chi^+(e) - \chi^-(e),
\end{equation}
with non-negative $\chi^\pm : \Delta_{\mathbb K}^1 \to \N \cup \{ 0 \}$ satisfying
\begin{equation}\label{mupmeqs}
\chi^\pm(\bar e)=\chi^\mp(e) \ \ \  \text{ and } \ \ \
\sum_{s(e)=v} \chi^+(e)=\sum_{s(e)=v} \chi^-(e),
\end{equation}
for all $e\in \Delta_{\mathbb K}^1$. One then considers the equations
\begin{equation}\label{mugpmtree}
g^\pm(r(e)) - g^\pm(s(e))=\chi^\pm(e).
\end{equation}
We first see that \eqref{mugpmtree} determines unique solutions
$g^\pm: \Delta_{\mathbb K}^0\to \Q_+$ with $g^\pm(v)=0$ at the basepoint.
In fact, suppose we are given a vertex $w\neq v$
in the tree. There is a unique path $P(v,w)$ in $\Delta_{\mathbb K}$ connecting
the base vertex $v$ to $w$. It is given by a sequence $P(v,w)= e_1,\ldots, e_n$ 
of oriented edges. Let $v=v_0,\ldots,v_n=w$ be the corresponding
sequence of vertices. The equation \eqref{mugpmtree} implies 
\begin{equation}\label{gpmind}
 g^\pm (w) = \sum_{j=1}^{n} \chi^\pm (e_j) .
\end{equation}
This determines uniquely the values of $g^\pm$ at each vertex in
$\Delta_{\mathbb K}$. The solutions obtained in this way satisfy
$g^+(v)-g^-(v)=g(v)$, where $g(v)=\log_q \| f 
\|_{\Lambda^{-1}(v)}$ as in Lemma
\ref{thetaweights}. 
The $g^\pm$ satisfy by construction the inhomogeneous weight equation
$$ g^\pm(w)+d\chi^\pm(v)=\sum_{s(e)=w} \lambda(e) g^\pm(r(e)), $$
for $\lambda(e)=N_e^{-1}$.
Thus, the pair $(g,\lambda)$ of Lemma
\ref{thetaweights} is a rational virtual weight.
\endproof

Notice that, even though the current $\mu_f$ is $\Gamma$-invariant 
by construction, and the function $\lambda(e)=N_{e}^{-1}$ is also
$\Gamma$-invariant, the function $g: \Delta_{\mathbb K}^0 \to \Q$ obtained
as above is not in general $\Gamma$-invariant, hence it need not
descend to a graph weight on $\Delta_{\mathbb K}/\Gamma$. 

In fact, for $g(v)=\log_q \| f \|_{\Lambda^{-1}(v)}$, one sees that 
$$ g(\gamma v)= \log_q \| f \|_{\Lambda^{-1}(\gamma v)} =
\log_q \| f\circ \gamma \|_{\Lambda^{-1}(v)} = \log_q |c(\gamma)| +
\log_q \| f \|_{\Lambda^{-1}(v)} = g(v) + \log_q |c(\gamma)| , $$
where $f(\gamma z)= c(\gamma) f(z)$.

More generally, one has the following result.

\begin{lemma}\label{noninvalpha}
Let $(g,\lambda)$ be a rational virtual weight on the tree $\Delta_{\mathbb K}$,
with $g: \Delta_{\mathbb K}^0 \to \Q$ and with $\lambda(e)=N_e^{-1}$. 
Then the function $g$ satisfies 
\begin{equation}\label{gbetagamma}
g(\gamma v)-g(v)= d\beta_\gamma (v),
\end{equation}
where $d\beta_\gamma (v)=\sum_{s(e)=v} \beta_\gamma(e)$ and
\begin{equation}\label{betagamma}
\beta_\gamma (e)= \lambda(e) (g(\gamma r(e)) - g(s(e))).
\end{equation}
This satisfies $\beta_\gamma(\bar e)=-\beta_{\gamma^{-1}}(\gamma e)$
and the 1-cocycle equation
\begin{equation}\label{1cocycle}
d\beta_{\gamma_1\gamma_2}(v)= d\beta_{\gamma_1}(\gamma_2 v)+
d\beta_{\gamma_2}(v). 
\end{equation}
\end{lemma}

\proof First notice that, by the weight equation \eqref{weighteq}, 
the function $g$ satisfies
\begin{equation}\label{galphagamma}
g(\gamma v)-g(v)= d\alpha_\gamma (v),
\end{equation}
where $d\alpha_\gamma (v)=\sum_{s(e)=v} \alpha_\gamma(e)$ with
\begin{equation}\label{alphagamma}
\alpha_\gamma (e)= \lambda(e) \left(\,\,g(\gamma r(e))-g(r(e))\,\,\right).
\end{equation}
Notice moreover that we have
$$ \alpha_\gamma(e)=\beta_\gamma(e) - \mu(e), $$
for $\beta_\gamma$ as in \eqref{betagamma} and $\mu=\mu_f$ the
$\Gamma$-invariant current \eqref{mufnorm}. Since $d\mu(v)\equiv 0$ we
have $d\alpha_\gamma (v)=d\beta_\gamma(v)$, which gives
\eqref{gbetagamma}. One checks the expression for $\beta_\gamma(\bar
e)$ directly from \eqref{betagamma}, using the $\Gamma$-invariance of
$N_e$ and $\lambda(e)$. The 1-cocycle equation is also easily verified
by $$ g(\gamma_1\gamma_2 v) - g(v) - g(\gamma_1\gamma_2 v)+ g(\gamma_2 v) -
g(\gamma_2 v)+ g(v) = 0. $$
\endproof

In general, the condition for a (virtual) graph weight on the tree
$\Delta_{\mathbb K}$ to descend to a (virtual) graph weight on the quotient
$ \Delta_{\mathbb K}/\Gamma$ is that the functions $(g,\lambda)$  satisfy 
\begin{equation}\label{weighteqgamma}
g(v)= \sum_{s(e)=\gamma v} \lambda(e) g(r(e)),
\end{equation}
for all $\gamma \in \Gamma$. This is clearly equivalent to the
vanishing of $d\beta_\gamma(v)$ and to the invariance $g(\gamma
v)=g(v)$. 

\medskip

Another possible way of describing (rational) virtual graph weights,
instead of using the inhomogeneous
equations, is by allowing the function $\lambda$ to have positive or
negative sign, namely we consider $\lambda: E^1 \to \Q$ and look
for non-negative solutions $g: E^0 \to \Q_+$ of the original graph weight 
equation \eqref{weighteq}.

A rational virtual weight $(g,\lambda)$ defines a solution
$(\tilde g, \tilde \lambda)$ as above, with $\tilde \lambda: E^1 \to
\Q$ and $\tilde g: E^0 \to \Q_+$, by setting 
$\lambda$ to be $\tilde\lambda(e) = \lambda(e)
\sign(g(s(e))) \sign(g(r(e)))$ and $\tilde g(v) = \sign(g(v))
g(v)=|g(v)|$. This definition has an ambiguity when $g(v)=0$, in which
case we can take either $\sign(g(v))=\pm 1$.

\subsection{Theta functions and $K$-theory classes}

Another useful observation regarding the relation of theta functions
of the Mumford curve $X_\Gamma$ to properties of the graph algebra of
the infinite graph $\Delta_{\mathbb K}/\Gamma$, is the fact that one
can associate to the theta functions elements in the $K$-theory of the
boundary $C^*$-algebra $C(\partial \Delta_{\mathbb K})\rtimes \Gamma$. 

This is not a new observation: it was described explicitly in
\cite{Rob} and also used in \cite{CLM}, though only in the case of
finite graphs. The finite graph hypothesis is used in \cite{Rob}
to obtain the further identification of the $\Gamma$-invariant
$\Z$-valued currents on the covering tree with the first homology
group of the graph. 

In our setting, the graph $\Delta_{\mathbb K}/\Gamma$ consists of a
finite graph $\Delta'_\Gamma/\Gamma$ together with infinite
trees stemming from its vertices. We still have the same result on the
identification with the $K$-theory group $K_0(C(\partial
\Delta_{\mathbb K})\rtimes \Gamma)$ of the boundary algebra, as well as with the
first homology of the graph $\Delta_{\mathbb K}/\Gamma$, which is the
same as the first homology of the finite graph $\Delta'_\Gamma/\Gamma$. 

\begin{prop}\label{thetaKth}
There are isomorphisms
\begin{equation}\label{currKthH1}
\cC(\Delta_{\mathbb K},\Z)^\Gamma \cong H_1(\Delta_{\mathbb K}/\Gamma,\Z) 
\cong \Hom(K_0(C(\partial \Delta_{\mathbb K})\rtimes \Gamma),\Z).
\end{equation}
\end{prop}

\proof The first isomorphism follows directly from \eqref{currGammaAB}.

To prove the second identification $$\cC(\Delta_{\mathbb K},\Z)^\Gamma \cong \Hom(K_0(C(\partial
\Delta_{\mathbb K})\rtimes \Gamma),\Z),$$ first notice that $\partial
\Delta_{\mathbb K}$ is a totally disconnected compact Hausdorff space, hence $K_1(C(\partial
\Delta_{\mathbb K}))=0$ and in the exact sequence of \cite{PV} for the $K$-theory
of the crossed product by the free group $\Gamma$ one obtains an
isomorphism of $K_0(C(\partial\Delta_{\mathbb K})\rtimes \Gamma)$ with the
coinvariants 
$$ C(\partial\Delta_{\mathbb K},\Z)_\Gamma = C(\partial\Delta_{\mathbb K},\Z)/\{ f\circ\gamma -
f \,|\, f \in C(\partial\Delta_{\mathbb K},\Z)\}, $$
where $C(\partial\Delta_{\mathbb K},\Z)$ is the abelian group of locally
constant $\Z$-valued functions on $\partial\Delta_{\mathbb K}$, \ie finite
linear combinations with integer coefficients of characteristic
functions of clopen subsets. 
We then show that the abelian group $\cC(\Delta_{\mathbb K},\Z)^\Gamma$ of
$\Gamma$-invariant currents on the tree $\Delta_{\mathbb K}$ can be identified
with 
\begin{equation}\label{currHomH1}
\cC(\Delta_{\mathbb K},\Z)^\Gamma \cong \Hom(K_0(C(\partial\Delta_{\mathbb K})\rtimes
\Gamma),\Z) = H_1(\Delta_{\mathbb K}/\Gamma,\Z).
\end{equation}
To see that a current $\mu \in \cC(\Delta_{\mathbb K},\Z)^\Gamma$ defines a
homomorphism $\phi: C(\partial\Delta_{\mathbb K},\Z)_\Gamma \to \Z$, we use the
fact that we can view the current $\mu$ on the tree as a measure $m$ of 
total mass zero on the boudary $\partial\Delta_{\mathbb K}$ by setting
$m(V(e))=\mu(e)$, where $V(e)$ is the subset of the boundary
determines by all infinite paths starting with the oriented edge $e$. 
We then define the functional 
\begin{equation}\label{muphi}
 \phi(f)=\int_{\partial\Delta_{\mathbb K}} f \, dm, 
\end{equation}
where the integration is defined by $\phi(\sum_i a_i \chi_{V(e_i)}) =
\sum_i a_i \mu(e_i)$ on characteristic functions. To see that $\phi$
is defined on the coinvariants it suffices to check that it vanishes
on elements of the form $f\circ \gamma - f$, for some $\gamma\in
\Gamma$. This follows by change of variables and the invariance of the
current $\mu$,
$$ \int f\circ \gamma \, dm = \int f \, dm\circ\gamma^{-1} 
= \int f \, dm. $$
Conversely, suppose we are given a homomorphism $\phi:
C(\partial\Delta_{\mathbb K},\Z)_\Gamma \to \Z$. We define a map $\mu :
\Delta_{\mathbb K}^1 \to \Z$ by setting $\mu(e)=\phi(\chi_{V(e)})$, where
$\chi_{V(e)}$ is the characteristic function of the set $V(e)\subset
\partial\Delta_{\mathbb K}$. We need to show that this defines a
$\Gamma$-invariant current on the tree. We need to check that the equation
$$ \sum_{s(e)=v} \mu(e) =0 $$
and the orientation reversal condition $\mu(\bar e)=-\mu(e)$ is
satisfied. 

Notice that we have, for any given vertex $v\in \Delta_{\mathbb K}^0$, 
$\cup_{s(e)=v} V(e) =\partial \Delta_{\mathbb K}$. If we set
$$ h(v):= \sum_{s(e)=v} \phi(\chi_{V(e)}), $$
we obtain a $\Gamma$-invariant $\Z$-valued function on the set of vertices
$\Delta_{\mathbb K}^0$, \ie a $\Z$-valued function on the vertices
$(\Delta_{\mathbb K}/\Gamma)^0$. In fact, we have 
$$ \phi(f\circ\gamma)=\phi(f) $$
by the assumption that $\phi$ is defined on the coinvariants
$C(\partial\Delta_{\mathbb K},\Z)_\Gamma$, hence
$$ h(\gamma v)= \sum_{s(e)=\gamma v} \phi(\chi_{V(e)})= \sum_{s(e)=
v}\phi(\chi_{V(e)}\circ \gamma) = h(v). $$
Since by construction $h = d\mu$, with $\mu(e)=\phi(\chi_{V(e)})$ and
$d: \cA(\Delta_{\mathbb K}^1/\Gamma)\to
\cH(\Delta_{\mathbb K}^0/\Gamma)$ as in \eqref{longexact}, it is in the
kernel of the map $\Phi$ of \eqref{mapPhi}. This means that
$$ \Phi(h)=\sum_{v\in \Delta_{\mathbb K}^0/\Gamma} h(v) =0, $$
but we know that
$$ h(v)=\sum_{s(e)=v} \phi(\chi_{V(e)})=\phi(\sum_{s(e)=v}
\chi_{V(e)})= \phi(\chi_{\partial \Delta_{\mathbb K}}) $$
so that the condition $\Phi(h)=0$ implies $h(v)=0$ for all $v$, \ie
$\phi(\chi_{\partial \Delta_{\mathbb K}})=0$. This gives
$$ \sum_{s(e)=v} \phi(\chi_{V(e)}) =0 $$
which is the momentum conservation condition for the measure
$\mu$. Moreover, the fact that the measure on
$\partial\Delta_{\mathbb K}$ defined by $\mu(e)=\phi(\chi_{V(e)})$ has total
mass zero also implies that
$$ 0=\phi(\chi_{\partial \Delta_{\mathbb K}}) = \phi(\chi_{V(e)})+
\phi(\chi_{V(\bar e)}), $$
hence $\mu(\bar e)=-\mu(e)$, so that $\mu$ is a current.
The condition $\phi(f\circ\gamma)=\phi(f)$ shows that it is a
$\Gamma$-invariant current.
\endproof

The results of this section relate the theta functions of Mumford curves to the $K$-theory
of a $C^*$-algebra which is not directly the graph algebra $C^*(\Delta_{\mathbb K}/\Gamma)$
we worked with so far, but the ``boundary algebra'' $C(\partial
\Delta_{\mathbb K})\rtimes \Gamma$.  However, it is known by the result of Theorem~1.2 of \cite{KuPa} that
the crossed product algebra $C(\partial \Delta_{\mathbb K})\rtimes \Gamma$ is strongly Morita 
equivalent to the algebra $C^*(\Delta_{\mathbb K}/\Gamma)$.  In fact, we
use the fact that $\Delta_{\mathbb K}$ is a tree and that the $p$-adic Schottky group 
$\Gamma$ acts freely on $\Delta_{\mathbb K}$, so that $C^*(\Delta_{\mathbb K})\rtimes \Gamma\simeq C^*(\Delta_{\mathbb K}/\Gamma) \otimes \cK(\ell^2(\Gamma))$. Thus $C^*(\Delta_{\mathbb K}/\Gamma)$ is strongly Morita
equivalent to $C^*(\Delta_{\mathbb K})\rtimes \Gamma$. Moreover, $\Delta_{\mathbb K}$ is the universal covering tree
of $\Delta_{\mathbb K}/\Gamma$ and $\Gamma$ can be identified with the fundamental group so that
the argument of Theorem 4.13 of \cite{KuPa} holds in this case 
and the result of Theorem~1.2 of  \cite{KuPa} applies.
Thus the $K$-theory considered here can be also thought of as the $K$-theory of the latter algebra
and we obtain the following result.

\begin{cor}\label{ThetaK}
A theta function $f\in \Theta(\Gamma)$ defines a functional 
$\phi_f\in \Hom(K_0(C^*(\Delta_{\mathbb K}/\Gamma)),\Z)$. Two theta functions $f,f'\in \Theta(\Gamma)$ define the same $\phi_f=\phi_{f'}$ if and only if they differ by the action of ${\mathbb K}^*$.
\end{cor}

\proof The first statement follows from Proposition \ref{thetaKth} and the 
identification
$K_0(C^*(\Delta_{\mathbb K}/\Gamma))=K_0(C(\partial\Delta_{\mathbb K})\rtimes \Gamma)$ 
which follows from the strong Morita equivalence discussed above. The 
second statement is then a direct consequence of
Proposition \ref{thetaKth} and \eqref{ThetaX}.
\endproof

Corollary \ref{ThetaK} shows that there is a close relationship between the
$K$-homology of $C^*(\Delta_{\mathbb K}/\Gamma)$ and theta functions. In the
next section we make a first step towards constructing theta functions from
graphical data and spectral flows.

\section{Inhomogenous graph weights equation and the spectral flow}

Let $E$ be a graph with zhyvot $M$ with a 
choice of (not necessarily special) graph weight $(g,\lambda)$ adapted to the zhyvot 
action. We would like to construct an inhomogenous graph weight $(G,\lambda,\chi)$ 
from these data.

The motivation for the construction is as follows. In the case of a special graph weight, 
the spectral flow only sees edges 
and paths in the zhyvot $M$ of $E$. Consequently
$$
-\sum_{s(e)=v}sf_{\phi_\D}(S_eS_e^*\D,S_e\D S_e^*)=
\sum_{\substack{s(e)=v\\ e\in M}}
\lambda(e) g(r(e))\leq g(v).
$$
Thus in some sense the spectral flow is trying to reproduce the graph weight, but
misses information from edges not in $M$. Alternatively, one may think of 
restricting $(g,\lambda)$ to the zhyvot and asking whether it is still a (special)
graph weight. This is usually not the case, for exactly the same reason.  

So to obtain $(G,\lambda,\chi)$ we begin with the ansatz that $\lambda$ is the 
function given to us with our graph weight and that
$$
G(v)=g(v)-\alpha(v),
$$
so that we require $\alpha(v)\leq g(v)$ for all $v\in E$. We now compute
\begin{align*} 
\sum_{s(e)=v}\lambda(e)G(r(e)) &= g(v)-\sum_{s(e)=v} \lambda(e)\alpha(r(e))\\
&=g(v)-\alpha(v)+\alpha(v)-\sum_{s(e)=v} \lambda(e)\alpha(r(e))\\
&=G(v)+\sum_{s(e)=v}\left(\frac{1}{N_e}\alpha(v)-\lambda(e)\alpha(r(e))\right).
\end{align*}
Hence to obtain an inhomogenous graph weight, we must have
$$
\chi(e)=\frac{1}{N_e}\alpha(s(e))-\lambda(e)\alpha(r(e)).
$$
Here are some possible choices of $\alpha$.

1) $\alpha=g$. This forces $G=0$ and so we obtain an inhomogenous graph
weight only when $\frac{1}{N_e}=\lambda(e)$.

2) $\alpha=c\,g$, $0<c<1$. Now $G\neq 0$ but we still need 
$\frac{1}{N_e}=\lambda(e)$ in order to obtain an inhomogenous graph weight.

3) $\alpha(v)=\sum_{\substack{s(e)=v\\ e\in M}}\lambda(e)g(r(e))
=-\sum_{s(e)=v}sf_{\phi_\D}(S_eS_e^*\D,S_e\D S_e^*)$. In this case we obtain 
$$
\chi(e)=\frac{1}{N_e}\sum_{\substack{s(f)=s(e)\\ f\in M}} \lambda(f)g(r(f))-\lambda(e)
\sum_{\substack{s(h)=r(e)\\ h\in M}} \lambda(h)g(r(h)).
$$
This {\bf may} be non-negative for certain values of $\lambda$.

4) $\alpha(v)=\left\{\begin{array}{ll} g(v) & v\in M\\ 0 & v\not\in M\end{array}\right.$
This gives
$$
\chi(e)=\left\{\begin{array}{ll} \frac{1}{N_e}g(s(e))-\lambda(e)g(r(e)) & s(e),\,r(e)\in M\\
\frac{1}{N_e}g(s(e)) & s(e)\in M,\, r((e)\not\in M\\
0 & s(e),\,r(e)\not\in M \end{array}\right.
$$
Thus $\chi$ is non-negative provided
$$
\frac{1}{N_e}g(s(e))\geq \lambda(e)g(r(e)).
$$

The choices 3) and 4) both yield triples $(G,\lambda,\chi)$ satisfying Equation
\eqref{inhomweight}, and all that remains to understand is the positivity of the 
function $\chi$.

In fact it is easy to construct examples where 4) fails to give a 
non-negative function
$\chi$. However, 3) is more subtle. At present we have no way of 
deciding whether we can 
always find a $g$ so that the function $\alpha$ in 3) is non-negative for $\lambda$ 
associated to the zhyvot action. It would seem that passing to field extensions allows us
to construct graph weights adapted to the (new) zhyvot so that both 3) and 4) fail. The 
reason is we may increase $N_e$ while keeping $\lambda(e)$ constant. 

We describe here another construction of inhomogenous graph weights adapted 
to the zhyvot action. This uses special graph weights, with $\lambda(e)\neq 1$ 
only on the edges inside the zhyvot, so it does not apply to the construction of
theta functions, where one needs $\lambda(e)=N_e^{-1}$ (which is $q^{-1}\neq 1$
outside of the zhyvot), but we include it here for its independent interest.

\begin{lemma}\label{lem:bigger-alpha} 
Let $(g,\lambda)$ be a graph weight on the graph $E$ with zhyvot $M$
such that $\lambda(e)\neq 1$ iff $e\in M^1$. 
With the notation of Section \ref{sec:mod-ind}, define $\alpha_k:E^0\to [0,\infty)$ by
$$
\alpha_k(v)=\phi_\D(p_v\Phi_k),\ \ \ k=0,1,2,\dots.
$$
Then $\alpha_{k-1}(v)\geq \alpha_k(v)$ for all $v\in E^0$.
\end{lemma}

\begin{proof}
We first observe that for $k\geq 1$, 
$$
\Phi_k=\sum_{|\mu|_\sigma=k}
\Theta_{S_\mu,S_\mu}.
$$ 
This 
follows easily by induction from $\Phi_0=\sum_{v\in E^0}\Theta_{p_v,p_v}$ and the definition 
of the zhyvot action. Then
\begin{align*}
\phi_\D(p_v\Phi_k)&=\sum_{e\in M^1,|\mu|_\sigma=k-1,s(\mu)=v}\lambda(\mu)\lambda(e)
Trace_\phi(\Theta_{S_\mu S_e, S_\mu S_e})\\
&=\sum_{e\in M^1,|\mu|_\sigma=k-1,s(\mu)=v}\lambda(\mu)\lambda(e)
\phi(S_e^*S_\mu^*S_\mu S_e)\\
&\leq \sum_{e\in E^1,|\mu|_\sigma=k-1,s(\mu)=v}\lambda(\mu)\lambda(e)
\phi(S_e^*S_\mu^*S_\mu S_e)\\
&= \sum_{e\in E^1,|\mu|_\sigma=k-1,s(\mu)=v}\lambda(\mu)
\phi(S_eS_e^*S_\mu^*S_\mu )\\
&= \sum_{|\mu|_\sigma=k-1,s(\mu)=v}\lambda(\mu)
\phi(p_{r(\mu)}S_\mu^*S_\mu)\\
&=\phi_\D(p_v\Phi_{k-1}).
\end{align*}
\end{proof}

\begin{thm} Let $(g,\lambda)$ be a graph weight on the graph $E$ with zhyvot $M$
such that $\lambda(e)\neq 1$ iff $e\in M^1$. Then for all $k\geq 1$ the triple
$(g-\alpha_k,\lambda,\chi_k)$ is an inhomogenous graph trace, where 
$\chi_k(e)=\lambda(e)(\alpha_{k-1}(r(e))-\alpha_k(r(e)))$.
\end{thm}

\begin{proof}
There are a couple of simple observations here. If $v\in E^0\setminus M^0$,
then $\alpha_k(v)=0$ when $k>0$. This follows from Lemma \ref{lem:bigger-alpha}.
This means that for $k\geq 1$
\begin{align*}
\sum_{s(e)=v}\lambda(e)\alpha_k(r(e))=\sum_{s(e)=v}\lambda(e)\phi_\D(p_{r(e)}\Phi_k)
&=\sum_{s(e)=v}\lambda(e)\phi_\D(S_e^*S_e\Phi_k)\\
&=\sum_{s(e)=v}\phi_\D(S_e\Phi_kS_e^*)\\
&=\sum_{s(e)=v}\phi_\D(S_eS_e^*\Phi_{k+1})\\
&=\phi_\D(p_v\Phi_{k+1})=\alpha_{k+1}(v),
\end{align*}
the last line following from the Cuntz-Krieger relations. Then the inhomogenous graph weight
equation is simple.
\begin{align*}
\sum_{s(e)=v}\lambda(e)(g(r(e))-\alpha_k(r(e)))
&=\sum_{s(e)=v}\lambda(e)(g(r(e))-\alpha_{k-1}(r(e)))+\sum_{s(e)=v}\lambda(e)(\alpha_{k-1}(r(e))
-\alpha_{k}(r(e)))\\
&=(g(v)-\alpha_k(v))+\sum_{s(e)=v}\lambda(e)(\alpha_{k-1}(r(e))-\alpha_{k}(r(e))).
\end{align*}
So the inhomogenous graph weight equation is satisfied if we set 
$\chi_k(e)=\lambda(e)(\alpha_{k-1}(r(e))-\alpha_k(r(e)))$, and both 
$g-\alpha_k,\,\chi_k\geq 0$ by Lemma
\ref{lem:bigger-alpha}.
\end{proof}

These inhomogenous weights are canonically associated to the decompositions 
$F=F_k\oplus G_k$ of the fixed point algebra arising from the zhyvot action.

\subsection{Constructing theta functions from spectral flows}

We show how some of the methods described above that produce inhomogeneous
graph weights with $\lambda(e) =N_e^{-1}$ can be adapted to construct theta functions 
on the Mumford curve.

First notice that, given such a construction of a solutions of the inhomogeneous
graph weights as above, we can produce a rational virtual graph weight
in the following way.

\begin{lemma}\label{2weights}
Suppose we are given two graph weights $g_1$, $g_2$ on the infinite graph 
$\Delta_{\mathbb K}/\Gamma$, both with the same $\lambda(e)=N_e^{-1}$. Suppose we are also given
inhomogeneous graph weights of the form $G_i(v)=g_i(v)-\alpha_i(v)$ as above,
with $0\leq \alpha_i(v)\leq g_i(v)$ at all vertices, and with
$\chi_i(e)= \frac{1}{N_e} (\alpha_i(s(e))-\alpha_i(r(e)))$. Then setting 
$\hat G_i(v)=g_i(v)-\alpha(v)$ with $\alpha(v)=\min\{\alpha_1(v),\alpha_2(v)\}$
gives two solutions of the inhomogeneous graph weight equation with the same
$\chi(e)= \frac{1}{N_e} (\alpha(s(e))-\alpha(r(e)))$ and $\lambda(e)=N_e^{-1}$.
\end{lemma}

\proof We have
$$ \hat G_i(v) = g_i(v)-\alpha(v) = \sum_{s(e)=v} \frac{1}{N_e} g_i(r(e)) - \alpha(v) =
\sum_{s(e)=v} \frac{1}{N_e} \hat G_i(r(e)) +
\sum_{s(e)=v} \frac{1}{N_e}(\alpha(r(e)) -\alpha(s(e)) $$
which shows that both $\hat G_i(v)$ are solutions of the inhomogeneous weight
equation 
$$ \hat G_i(v) + d\chi(v) = \sum_{s(e)=v} \frac{1}{N_e} \hat G_i(r(e)) , $$
where $\chi(e)= N_e^{-1} (\alpha(s(e)) -\alpha(r(e)))$. 
\endproof

Thus, whenever we have multiple solutions for the graph weights on $\Delta_{\mathbb K}/\Gamma$
we can construct associated rational virtual graph weights by setting
\begin{equation}\label{ratvirGhat}
\hat G(v) = \hat G_1(v) - \hat G_2(v).
\end{equation}
If the graph weights $g_i$ are rational valued, $g_i: E^0 \to \Q_+$, then the virtual
graph weight $\hat G$ is also rational valued, $\hat G: E^0 \to \Q$.

Moreover, since the graph $\Delta_{\mathbb K}/\Gamma$ has a finite zhyvot with infinite
trees coming out of its vertices, one can obtain a rational virtual graph weight 
that is integer valued. In fact, along the
trees outside the zhyvot $\Delta_\Gamma^\prime/\Gamma$ of $\Delta_{\mathbb K}/\Gamma$,
the condition
$$ \hat G(v) = \sum_{s(e)=v} \frac{1}{N_e} \hat G(r(e)) $$
is satisfied by extending $\hat G(v)$ from the zhyvot by 
$\hat G(r(e))=\hat G(s(e))$ along the trees.  In this case, what remains is 
the finite graph, the zhyvot, which only involves finitely many denominators
for a rational valued $\hat G(v)$, which means that one can obtain an integer
valued solution. We will therefore assume that the rational virtual graph
weights constructed as in \eqref{ratvirGhat} and Lemma \ref{2weights} are
integer valued, $\hat G: E^0 \to \Z$. This in particular includes the cases constructed 
using the spectral flow.

According to the results of \S \ref{weightthetasec} above, we then have the
following result.

\begin{prop}\label{thetafromflow}
Suppose we are given a virtual graph weight $\hat G: E^0(\Delta_{\mathbb K}/\Gamma)\to \Z$ as above and a
homomorphism $c: \Gamma \to {\mathbb K}^*$. Then there is a theta function $f$ on
the Mumford curve $X_\Gamma$ satisfying $\log_q \| f \|_{\Lambda^{-1}(v)} =
\tilde G(v)$ and $f(\gamma z)=c(\gamma)f(z)$, where $\tilde G: E^1(\Delta_{\mathbb K})\to \Z$
is defined as $\tilde G(v)=\hat G(v)$ on a fundamental domain of the action of $\Gamma$
on $\Delta_{\mathbb K}$ and extended to $\Delta_{\mathbb K}$ by $\tilde G(\gamma v) = \log_q |c(\gamma)| 
+ \hat G(v)$.
\end{prop}

\proof This follows from the identification of the group of theta functions $\Theta(\Gamma)$
with the extension \eqref{ThetaX} of the group $\cC(\Delta_{\mathbb K})^\Gamma$ of currents on
$\Delta_{\mathbb K}/\Gamma$ by ${\mathbb K}^*$, and identifying the current
$\mu(e)= \hat G(r(e)) - \hat G(s(e))$ with $\mu_f(e)=\log_q \| f\|_{\Lambda^{-1}(r(e))}-
\log_q \| f\|_{\Lambda^{-1}(s(e))}$. 
\endproof

The last two results highlight the interest in determining whether non-negative $\alpha$ can be 
found for the function $\lambda(e)=N_e^{-1}$.

\end{document}